\documentclass[11pt,twoside]{amsart}
\usepackage{amssymb}
\usepackage{graphicx,color}

\newtheorem{thm}{Theorem}[section]
\newtheorem{prop}[thm]{Proposition}

\newtheorem{lem}[thm]{Lemma}
\newtheorem{cor}[thm]{Corollary}
\newtheorem{defn}[thm]{Definition}
\newtheorem{notation}[thm]{Notation}
\newtheorem{ex}[thm]{Example}
\newtheorem{rem}[thm]{Remark}

\newtheorem{prob}[thm]{Problem}

\newcommand{\skipit}[1]{{}}
\newcommand{\prfend}{\hbox to7pt{\hfil}
\par\vskip-\baselineskip\hbox to\hsize
{\hfil\vbox {\hrule width6pt height6pt}}\vskip\baselineskip}

\newcommand{\ZZ}{\mathbb{Z}}

\newcommand {\PP}{\mathbb{P}}

\newcommand{\W}{\mathcal{W}_{d}}

\DeclareMathOperator{\Image}{Im}

\newcommand{\myarrow}[2]{\hbox to #1pt{\hfil$\to$\hfil}{\hskip-#1pt{\raise
10pt\hbox to#1pt{\hfil$\scriptscriptstyle #2$\hfil}}}}

\begin{document}

\title{Minimal set of binomial generators for certain Veronese $3$-fold projections.}

\author[Liena Colarte]{Liena Colarte}
\address{Department de matem\`{a}tiques i Inform\`{a}tica, Universitat de Barcelona, Gran Via de les Corts Catalanes 585, 08007 Barcelona,
Spain}
\email{liena.colarte@ub.edu}

\author[Rosa M. Mir\'o-Roig]{Rosa M. Mir\'o-Roig}
\address{Department de matem\`{a}tiques i Inform\`{a}tica, Universitat de Barcelona, Gran Via de les Corts Catalanes 585, 08007 Barcelona,
Spain}
\email{miro@ub.edu}

\begin{abstract} The goal of this paper is to explicitly describe a minimal binomial generating set of a class of lattice ideals, namely the ideal of certain Veronese $3$-fold projections. More precisely, for any integer $d\ge 4$ and any $d$-th root $e$ of 1 we denote by $X_d$ the toric variety defined as the image of the morphism $\varphi _{T_d}:\PP^3 \longrightarrow \PP^{\mu (T_d)-1}$ where $T_d$ are all monomials of degree $d$ in $k[x,y,z,t]$ invariant under the action of the diagonal matrix  $M(1,e,e^2,e^3).$ In this work, we  describe a $\ZZ$-basis of the lattice $L_{\eta }$ associated to $I(X_d)$  as well as a minimal binomial set of generators of the lattice ideal $I(X_d)=I_+(\eta)$.
\end{abstract}

\thanks{Acknowledgments:   The   authors are partially   supported
by  MTM2016--78623-P.
\\ {\it 2010 Mathematic Subject Classification. 13E10, 14M25, 14N05, 14N15, 53A20.}}

\maketitle

\tableofcontents

\markboth{L. Colarte, R. M. Mir\'o-Roig}{Generators of lattice ideals}

%*****************************************************************************
\large

\section{Introduction}

A binomial ideal  $I\subset k[x_{0},\hdots, x_{n}]$ with $k$ a field
is an ideal generated by polynomials with at most two summands, say $ax^{\alpha}+bx^{\beta}$, where $a,b \in k$ and $\alpha,\beta \in
\ZZ_{+}^{n+1}$. Binomial ideals are a large class of ideals which have been amply studied in Combinatoric, Commutative Algebra as well as in Algebraic Geometry. In \cite{ES}, it was stated that prime binomial ideals are precisely the defining ideals of toric varieties and hence they are lattice ideals, i.e. given a prime binomial ideal $I\subset k[x_{0},\hdots, x_{n}]$ there is a lattice $L\subset \ZZ^{n+1}$ such that  $I=I_L:=\{x^u-x^v \mid u,v \in \ZZ^{n+1} \text{ and }u-v\in L \}$.
Ever since, to compute explicitly a minimal set of generators for lattice ideals has been a challenging problem. It is worthwhile to point out that  for a given generating set ${\mathcal D}$ of the lattice $L$ the ideal $I({\mathcal D}) = (x^{\alpha^{+}}-x^{\alpha_{-}} \, \mid \, \alpha_{+},\alpha_{-} \in \ZZ_{+}^{n+1}, \alpha_{+}-\alpha_{-} \in {\mathcal D}) \subseteq
I_L$ and the equality does not hold in general.

In \cite{ES}, Eisenbud and Sturmfels proved that $I_L$ is a prime ideal if and only if the lattice $L$ is saturated. For  prime binomial ideals, a set of generators ${\mathcal D}$ of $L$ completely characterizes a set of generators of $I_L$. Indeed, a generating set ${\mathcal D}$ of $L$ is called a Markov basis if for any lattice point $\alpha_{+}-\alpha_{-} \in L$ there exits a finite sequence $\{a_{1},\hdots, a_{t}\} \subset \ZZ^{n+1}_{+}$ satisfying $a_{1} = \alpha_{+}, a_{t} = \alpha_{-}$ and $a_{i}-a_{i+1} \in \mathcal{D}$ for all $1 \leq i \leq t$. In \cite{DST}, Diaconis and Sturmfels showed that given a set of generators ${\mathcal D}$ of $L$ then $I({\mathcal D}) = I_L$ if and only if ${\mathcal D}$ is a Markov basis. We cite \cite{ChaKT}, \cite{ChaTV}, \cite{ChaTV1}, \cite{DST} and \cite{HM} for a detailed exposition of Markov bases of lattice ideals and related problems.

In this paper, we focus our attention in computing a minimal binomial set of generators of a large family of  binomial ideals $I(X_d)$. They are the ideals associated to  suitable projections of Veronese 3-folds.
A Veronese 3-fold $V$ is a projective variety given parametrically by the set $M_{3,d}$ of all monomials of degree $d$ in $k[x_0,x_1,x_2,x_3]$ and by  a projection of $V$ we understand a projective 3-fold given parametrically by a subset of $M_{3,d}$. In \cite{Grob},  Gr\"{o}bner proved that $V$ is arithmetically Cohen-Macaulay (aCM, for short)
and its ideal $I(V)$ is generated by quadrics. This is not longer true for all projections of $V$ and it is a longstanding open problem to find a minimal set of generators of any projection of $V$ as well as determine whether a projection of $V$ is  aCM. In this paper, we will consider as a subset of $M_{3,d}$ the set $T_{d}$ of all monomials of degree $d$ invariant under the action of the diagonal  matrix $M(1,e,e^2,e^3)$ where $e$ is a primitive root of 1 of order $d$.

 Our interest in these ideals $T_d$ relies on  the following three facts: (1) For all $d\ge 4$ $T_d$ fails the Weak Lefschetz property (WLP) in degree $d-1$; (2) The associated morphism  $\varphi_{T_{d}}: \PP^{3} \longrightarrow \PP^{\mu(T_{d})-1}$  is a Galois cover of degree $d$ with cyclic Galois  group $\ZZ/d\ZZ$ and the image  $X_d$ of $\varphi_{T_{d}}$ is a $3$-dimensional rational projective variety smooth outside the image of the 4 fundamental points. We call it a GT-threefold; and (3) the 3-fold
 $Y_d=\overline{\Image (
\phi )}$
where
$\phi \colon\PP^n \dashrightarrow \PP^{\binom{3+d}{d}-\mu (T_d)-1}$ is the  rational map associated to $(I^{-1})_d$,
  satisfies at least one Laplace equation of order
$d-1$.

Our goal is to prove  that the homogeneous ideal  $I(X_{d})$ of the $GT$-threefold $X_d$ is the homogeneous prime binomial ideal associated to a saturated partial character $\eta$ of $\ZZ^{\mu(T_{d})}$ with associated lattice $L_{\eta}$. Afterwards we explicitly compute a minimal binomial set of generators of $I(X_{d})$. The lattice points associated to these set of generators form a Markov basis of $L_{\eta}$. Our main result states that $I(X_{d})$ is generated by quadrics if $d$ is even and by quadrics and cubics if $d$ is odd.

\vspace{0.3cm}
Next we outline the structure of this note. In Section 2, we fix the notation we use in the rest of this paper, we relate artinian ideals failing  the Weak Lefschetz Property to projective varieties satisfying at least one Laplace equation and we recall the  notion of  Togliatti systems and $GT$-systems introduced in \cite{MMO} and \cite{MeMR}. In Section 3,  we give an explicit description of all monomials $T_{d}$, $d \geq 4$, invariant under the action of the diagonal matrix $M(1,e,e^{2},e^{3})$ and we prove that $T_d$ is a GT-system (Proposition \ref{Prop:GaloisCover}).

The main body of this work is developed in Sections 4 and 5. We denote by $X_d$ the GT-threefold  associated to the GT-system $T_d$ and we first show that $X_{d}$ is an irreducible toric variety whose associated ideal $I(X_{d})$ is a lattice ideal. In section 4, we consider the
ideal $I_{d}$ generated by all  binomials of degree $2$ vanishing in $X_{d}$. We associate to $I_{d}$ a lattice $L_{\eta}$ and a partial character $\eta$ of $\ZZ^{\mu(T_{d})}$. We demonstrate that $L_{\eta}$ is a saturated lattice of rank $\mu(I_{d})-4$ (Theorem \ref{Thm:MainTheorem}) and we show that $I(X_{d})$ is the lattice ideal $I_{+}(\eta)$ of $L_{\eta}$ (Corollary \ref{igualtat_ideals}). In Section 4, we also  describe a $\ZZ$-basis of the lattice $L_{\eta}$ (Corollary \ref{Col:Basis}) and we  explore the relation between $I_{d}$ and $I_{+}(\eta)$.

We devote Section 5 to explicitly determine a minimal set of generators of the lattice ideals
$I(X_{d})$. Our main result states that $I_{d} = I(X_{d})$ if $d$ is even and $I(X_{d}) = I_{d} + J$ if $d$ is odd where $J$ is an ideal generated by certain set of cubics of $I(X_{d})$ that we properly specify (Theorem \ref{Teorema:BigTheorem}). All techniques and results we develop to study the lattice ideal $I(X_{d})$ are inspired by the ones of Markov basis explained in \cite{DST}, \cite{HM} and \cite{ChaTV1}. The set of lattice points of generators of $I_{d}$ if $d$ even and $I_{d}$ and $J$ if $d$ odd forms a Markov basis of $L_{\eta}$. In Section 6, we observe that all $GT$-varieties are aCM and we concern about computing a minimal free resolution of $X_{d}$. 

\vskip 2mm \noindent
 {\bf Acknowledgement.}  The authors would like to thank E. Mezzetti and M. Salat for useful discussions on Galois-Togliatti systems.

%%%%%%%%%%%%%%%%%%%%%%%%%%%%%%%%%%%%%%%%%%%%%%%%%%
\section{Preliminaries}
\label{defs and prelim results}
Throughout this paper we consider the homogeneous polynomial ring $R = k[x_0,\cdots ,x_n]$ where $k$ is an algebraically closed field of characteristic zero.
Let $I\subset R $ be a homogeneous artinian ideal. We say that   $I$ has the \emph{Weak Lefschetz Property (WLP)}
if there is a linear form $L \in (R/I)_1$ such that, for all
integers $j$, the multiplication map
\[
\times L: (R/I)_{j-1} \to (R/I)_j
\]
has maximal rank, i.e.\ it is injective or surjective.   Though many homogeneous artinian ideals are expected to have the
WLP, establishing this property is often rather difficult.
 Recently the failure of the WLP has been connected to a large number of problems which  seem to be unrelated at first glance.
For example, in \cite{MMO}, Mezzetti, Mir\'{o}-Roig and Ottaviani proved
that the failure of the WLP is related to the existence of varieties satisfying at least one Laplace
equation of order greater than 2. More precisely, they proved:

\begin{thm} \label{tea} Let $I\subset R$ be an artinian
ideal
generated
by $r$ forms $F_1,\dotsc,F_{r}$ of degree $d$ and let $I^{-1}$ be its Macaulay inverse system.
If
$r\le \binom{n+d-1}{n-1}$, then
  the following conditions are equivalent:
\begin{itemize}
\item[(1)] $I$ fails the WLP in degree $d-1$;
\item[(2)]  $F_1,\dotsc,F_{r}$ become
$k$-linearly dependent on a general hyperplane $H$ of $\PP^n$;
\item[(3)] the $n$-dimensional   variety
 $X=\overline{\Image (
\varphi )}$
where
$\varphi \colon\PP^n \dashrightarrow \PP^{\binom{n+d}{d}-r-1}$ is the  rational map associated to $(I^{-1})_d$,
  satisfies at least one Laplace equation of order
$d-1$.
\end{itemize}
\end{thm}
\begin{proof} See \cite[Theorem 3.2]{MMO}.
\end{proof}

Motivated by  the above results,   Mezzetti, Mir\'{o}-Roig and Ottaviani  introduced the following definitions (see \cite{MMO} and \cite{MeMR}):

\begin{defn} \rm Let $I \subset R$ be an artinian ideal generated by $r$ forms  of degree $d$, and $r \leq {n+d-1\choose n-1}$. We will say:

\begin{itemize}
\item[(i)] $I$ is a \emph{Togliatti system} if it fails the WLP in degree $d-1$.
\item[(ii)] $I$ is a \emph{monomial Togliatti system} if, in addition, $I$ can be generated
by monomials.
\item[(iii)] $I$ is a \emph{smooth Togliatti system} if, in addition, the rational variety $X$ is smooth.
\item[(iv)] A monomial Togliatti system $I$ is \emph{minimal} if  there is no proper subset of the set of generators defining a monomial Togliatti system.
\end{itemize}
\end{defn}

The names are in honor of  Togliatti who classified all
 rational surfaces parameterized by cubics and satisfying at least one Laplace equation of
order $2$ and he proved that for
$n = 2$ the only smooth
Togliatti system of cubics is $$I = (x_0^3,x_1^3,x_2^3,x_0x_1x_2)\subset k[x_0,x_1,x_2]$$
(see \cite{BK}, \cite{T1} and \cite{T2}). The systematic study of Togliatti systems was initiated in \cite{MMO}
and for recent results the reader can see \cite{MMR},  \cite{MeMR}, \cite{AMRV}, \cite{MRS} and \cite{MeMR2}.
Precisely in the latter reference the authors introduced the notion of \emph{GT-system} which we recall now.

\begin{defn}\label{Defi:GT}
  A \emph{GT-system} is an artinian ideal $I\subset R$ generated by $r$ forms $F_{1},\dotsc,F_{r}$ of degree $d$
  such that:
\begin{enumerate}
\item[i)] $I$ is a Togliatti system.
\item[ii)] The regular map $\phi_{I}\colon\PP^{n}\rightarrow \PP^{r-1}$ defined by $(F_{1},\dotsc,F_{r})$ is a Galois covering of degree $d$ with cyclic Galois group $\ZZ /d\ZZ$.
\end{enumerate}
\end{defn}

Any representation of the cyclic group $\ZZ  /d\ZZ$ as subgroup of $GL(n+1,k)$ can be diagonalized.
In particular it is represented by a diagonal matrix of the form
$$M_{\alpha_0,\alpha_1,\dotsc,\alpha_{n}}=\begin{pmatrix} e^{\alpha _0} & 0 & \dotsc & 0 \\ 0 & e^{\alpha _1} &  \dotsc & 0 \\
 &  & \dotsc &  \\ 0 & 0 &\dotsc & e^{\alpha _{n}}
\end{pmatrix}$$
where $e$ is a primitive $d$th root of $1$ and $\alpha_0,\alpha_1,\dotsc,\alpha_{n}$ are integers with
$$GCD(\alpha_0,\alpha_1,\dotsc,\alpha_{n},d)=1.$$ %we  have
It follows (see \cite[Proposition 4.6]{CMMRS}) that the above definition is equivalent to the next one:

\begin{defn} \label{GT} \rm
Fix integers $3\le d\in \ZZ $, $2\le n\in \ZZ $, with $n\leq d$, and $0\le \alpha _0\le \alpha_1\le \cdots \le \alpha_{n}\le  d$, $e$ a primitive $d$-th root of 1 and $ M_{\alpha_0,\alpha_1,\cdots,\alpha_{n}}$ a representation of $\ZZ /d\ZZ$ in $GL(n+1,k)$. A \emph{GT-system} will be an ideal $$I^d_{\alpha_0,\cdots,\alpha _{n}}\subset R$$ generated by all forms of degree $d$ invariant under the action of   $ M_{\alpha_0,\alpha_1,\dotsc,\alpha_{n}}$ provided the number of generators $\mu (I^d_{\alpha_0,\dotsc,\alpha _{n}})\le  \binom{n+d-1}{n-1}$.
\end{defn}

Finally, note that the ideal $I^d_{\alpha_0,\dotsc,\alpha_n}$ is always monomial,
i.e. a GT-system is a monomial Togliatti system.

%%%%%%%%%%%%%%%%%%%%%%%%%%%%%%%%%

\section{GT-systems  and GT-varieties}
\label{GT-systems}

Through this section we fix an integer $d \geq 4$, a $dth$-root of unity $e$ and we write $d = 2k+\varepsilon = 3k' + \rho$ with $\varepsilon \in \{0,1\}$ and $\rho \in \{0,1,2\}$. We denote $T_{d} \subset R=k[x,y,z,t]$ the ideal generated by the $\mu(T_{d})$ monomials of degree $d$ invariant under the action of the diagonal matrix $M(1,e,e^{2},e^{3}).$
In this section, we will describe the ideal $T_{d}$ and we will prove that $T_{d}$ is a $GT$-system for all $d \geq 4$. We also define the $GT$-varieties $X_d$ and their apolar varieties $Y_d$. The homogeneous ideal $I(X_d)$ of a $GT$-variety $X_d$ is a lattice ideal. A basis of the lattice and a system of generators of the lattice ideal will be effectively computed in next sections.

A monomial  $x^{\alpha}y^{\beta}z^{\delta}t^{\gamma} \in R$  of degree $d$ belongs to $T_d$ if it is invariant under the action of
$M(1,e,e^2,e^3)$ or,  equivalently if  $\alpha ,\beta , \delta, \gamma $ satisfy:
$$(*) \quad \left.\begin{matrix}
   \alpha & + & \beta& + &\delta& + &\gamma& = &d\\
          &   & \beta&+ &2\delta& + &3\gamma& = &rd
  \end{matrix} \right\},\;\; r = 0,1,2,3.$$

The solutions of $(*)$ in terms of $\gamma$ and $r$ are the following:
$$ \begin{array}{ccl}
   \alpha & = &\delta + 2\gamma + (1-r)d,\\
   \beta  & = &rd - 2\delta - 3\gamma,\\
   \gamma & \in  &\{0, \hdots, rk'+ \lfloor \frac{r\rho}{3} \rfloor\},\\
   \delta & \in &\{max\{0, (r-1)d-2\gamma\}, \hdots, \lfloor \frac{rd-3\gamma}{2}\rfloor\}.
  \end{array}$$
Given $d \geq 4$, we define
$$\W := \{(r,\gamma,\delta) \in \ZZ^{3} \mid 0 \leq r \leq 3, 0 \leq \gamma \leq rk' + \lfloor \frac{r\rho}{3} \rfloor, max\{0,d-2\gamma\} \leq \delta \leq \lfloor \frac{rd-3\gamma}{2}\rfloor\}.$$
 All monomials  $x^{\alpha}y^{\beta}z^{\delta}t^{\gamma} \in T_d$  of degree $d$  are  uniquely determined by a triple $(r,\gamma,\delta) \in \W$. In particular, 
$\mu (T_d)= \#  \W.$

\begin{rem} \rm Notice that $x^{d},y^{d},z^{d}$ and $t^{d}$ are invariant under the action of $M(1,e,e^2,e^3)$. So, the ideal $T_{d}$ is artinian.
\end{rem}

In next example,
we explicitly exhibit $T_d$ for $d = 4,5,6,7,8$ and $9$. For these values of $d$ we cover all possibilities of $\varepsilon$ and $\rho$.

\begin{ex} \rm \label{Ej:4GeneIdeal}
\noindent $\begin{array}{rcl} T_{4}& = &(x^4, y^4, xy^2z, x^2z^2, x^2yt, z^4, y z^2 t,y^2t^2, xzt^2, t^4), \;\;\mu_{I_{4}} = 10. \end{array}$

\noindent $\begin{array}{rcl}T_{5}&=& (x^5, y^5, x y^3 z, x^2 y z^2, t x^2 y^2, t x^3 z, z^5, t y z^3, t^2 y^2 z, t^2 x z^2, t^3 x y, t^5),\;\;
\mu_{I_{5}} = 12.\end{array}$

\noindent $\begin{array}{rcl}T_{6}&=&(x^6, y^6, x y^4 z, x^2 y^2 z^2, x^3 z^3, t x^2 y^3, t x^3 y z
, t^2 x^4, z^6, t y z^4, t^2 y^2 z^2, t^2 x z^3, t^3 y^3, t^3 x y z, \\ & & t^4 x^2, t^6),\;\;  \mu_{I_{6}} = 16. \end{array}$

\noindent $\begin{array}{rcl}T_{7}& = &(x^7, y^7, x y^5 z, x^2 y^3 z^2, x^3 y z^3, t x^2 y^4, t x^3 y^2 z, t x^4 z^2, t^2 x^4 y, z^7, t y z^5, t^2 y^2 z^3, t^2 x z^4, t^3 y^3 z,\\
& & t^3 x y z^2, t^4 x y^2, t^4 x^2 z, t^7),\;\;\mu_{I_{7}} = 18.\end{array}$

\noindent $\begin{array}{rcl}T_{8}& =& (x^8, y^8, x y^6 z, x^2 y^4 z^2, x^3 y^2 z^3, x^4 z^4, t x^2 y^5, t x^3 y^3 z, t x^4 y z^2, t^2 x^4 y^2, t^2 x^5 z, z^8, t y z^6, t^2 y^2\!z^4\!, \\
& & t^2 x z^5,t^3 y^3 z^2, t^3 x y z^3, t^4 y^4, t^4 x y^2 z, t^4 x^2 z^2, t^5 x^2 y,t^8),\;\; \mu_{I_{8}} = 22.\end{array}$

\noindent$\begin{array}{rcl}T_{9}& =&(x^9, y^9, x y^7 z, x^2 y^5 z^2, x^3 y^3 z^3, x^4 y z^4, t x^2 y^6, t x^3 y^4 z, t x^4 y^2 z^2, t x^5 z^3, t^2 x^4 y^3, t^2 x^5 y z, t^3 x^6, \\
& & z^9, t y z^7,
t^2 y^2 z^5, t^2 x z^6, t^3 y^3 z^3, t^3 x y z^4, t^4 y^4 z, t^4 x y^2 z^2, t^4 x^2 z^3, t^5 x y^3, t^5 x^2 y z, t^6 x^3,t^9), \\ & & \mu_{I_{9}} = 26. \end{array}$
\end{ex}

Our interest in the study of these monomial ideals relies in the following fact:

\begin{prop}\label{Prop:GaloisCover} For any $d \geq 4$,  $T_{d}$ is a $GT-$system. In particular, $T_d$ fails the WLP in degree $d-1$.
\end{prop}
\begin{proof} By Definition \ref{GT}, we only have to check that $\mu(T_{d}) \leq \binom{2+d}{2}$. From the definition of $T_{d}$, it follows that
$$\mu(T_{d}) = 2 + \sum_{r=1,2}\sum_{\gamma=0}^{rk'+\lfloor \frac{r\rho}{3}
\rfloor}(\lfloor \frac{rd-3\gamma}{2}\rfloor - max\{0,(r-1)d-2\gamma\}+1).$$
We sum separately for $r = 1$ and
$r=2$; we have
$$
\begin{array}{rcl}
\sum_{\gamma=0}^{k'} (k - \lceil \frac{3\gamma-\varepsilon}{2}\rceil + 1) &= &(k'+1)(k+1) - \sum_{\gamma=1}^{k'} \lceil
\frac{3\gamma-\varepsilon}{2} \rceil, \text{ and}\\
\sum_{\gamma=0}^{2k' + \lfloor \frac{2\rho}{3} \rfloor}( d-\lceil \frac{3\gamma}{2} \rceil + 1) - \sum_{\gamma=0}^{k}(d -
2\gamma) &=& (d+1)(2k' + \lfloor \frac{2\rho}{3} \rfloor+1) +k(k+1)-\\
&&-d(k+1)- \sum_{\gamma=0}^{2k' + \lfloor \frac{2\rho}{3} \rfloor}
\lceil \frac{3\gamma}{2} \rceil.
\end{array}
$$
We only have to focus on the sum of the series of the type $\sum_{\gamma=1}^{N} \lceil \frac{3\gamma-\varepsilon}{2} \rceil$ with
$\varepsilon \in \{0,1\}$. We can rewrite the series as follows: if $N = 2j$,
$\sum_{\gamma=1}^{N}\lceil \frac{3\gamma-\varepsilon}{2} \rceil = \sum_{i=1}^{j} 3i + \sum_{i=1}^{j}(3i-1-\varepsilon) = j(3j + 2 - \varepsilon).$ Otherwise $N = 2j+1$,
$\sum_{\gamma=1}^{N}\lceil \frac{3\gamma-\varepsilon}{2} \rceil =\sum_{i=1}^{j}3j + \sum_{i=1}^{j+1}3j-1-
\varepsilon = (j+1)(3j + 2 -\varepsilon)$. In any case,
$$\sum_{\gamma=1}^{N}\lceil \frac{3\gamma-\varepsilon}{2}\rceil = \lceil \frac{N}{2} \rceil(3\lfloor \frac{N}{2} \rfloor + 2 - \varepsilon).$$
From this, we conclude
$$
\begin{array}{rcl}
\mu(T_{d})& =& 2 + (k'+1)(k+1) + (d+1)(2k' + \lfloor \frac{2\rho}{3} \rfloor+1)+k(k+1) -\\
&&-d(k+1) - \lceil \frac{k'}{2} \rceil(3\lfloor \frac{k'}{2} \rfloor + 2 - \varepsilon) -
\lceil \frac{2k' + \lfloor \frac{2\rho}{3} \rfloor}{2}  \rceil(3\lfloor \frac{2k' + \lfloor \frac{2\rho}{3} \rfloor}{2} \rfloor + 2).
\end{array}
$$
Substituting $d = 3k' +\rho$ by $k = \frac{3k'+\rho-\varepsilon}{2}$ we  verify that $\mu(T_{d}) \leq 2 + (k'+1)(\frac{3k'+\rho}{2} + 1)
+ (3k'+\rho+1)(2k'+2) + \frac{3k'+\rho}{2}(\frac{3k'+\rho}{2}+1)-(3k'+\rho)(\frac{3k'+\rho}{2} + 1) - \frac{k'}{2}(\frac{3(k'-1)}{2} + 1) - k'(3k'+2)
= \frac{1}{4} (20 + 6 (k')^2 + 8 \rho - \rho^2 + k' (29 + 4 \rho))$. It holds that $\frac{1}{4} (20 + 6 (k')^2 + 8 \rho - \rho^2 + k' (29 + 4 \rho)) <
\frac{1}{2}(d+2)(d+1) \Leftrightarrow 1/4 (16 - 12 (k')^2 + k' (11 - 8 \rho) + 2 \rho - 3 \rho^2) \leq 0$, which holds
 for all  $d \geq 4$.
\end{proof}

We finish this section studying the geometric properties
of the rational 3-fold associated to the GT-system $T_d$.
The morphism $\varphi_{T_{d}}: \PP^{3} \longrightarrow \PP^{\mu(T_{d})-1}$ associated to the GT-system $T_d$ is a Galois cover of degree $d$ with cyclic Galois  group $\ZZ/d\ZZ$ represented by $M(1,e,e^2,e^3)$.
In particular,  a general fibre of $\varphi_{T_{d}}$ consists of $d$ points, and hence the image of $\varphi_{T_{d}}$ is a $3$-dimensional rational projective variety.

\begin{defn} \rm  We call {\em GT-variety} and we denote it by $X_{d}$ the rational 3-fold defined as the image of $\varphi_{T_{d}}$.
\end{defn}

The morphism
$\varphi_{T_{d}}: \PP^{3} \longrightarrow \PP^{\mu(T_{d})-1}$ is unramified outside the four fundamental points of
$\PP^3$: $E_0=[1,0,0,0]$, $E_1=[0,1,0,0]$, $E_2=[0,0,1,0]$ and $E_4=[0,0,0,1]$. They are sent by
$\varphi_{T_d}$ to the singular points of $X_d$, $P_i:=\varphi(E_i)$,
$i=0,1,2,3$, that are cyclic quotient singularities: $P_0$ is of type $\frac{1}{d}(1,2,3)$, $P_1$ is of type $\frac{1}{2}(1,d-1,d-2)$, $P_2$ is of type $\frac{1}{d}(d-2,d-1,1)$ and
$P_3$ is  of type $\frac{1}{d}(d-3,d-2,d-1)$.

\begin{rem} \rm
(1)
It is worthwhile to point out that the rational 3-fold $X_d$ is also a Galois covering of $\PP^3$ with Galois
group $\ZZ/d\ZZ$. The covering map $X_d\to \PP^3$ composed with $\varphi_{T_d}$
is  $\PP^3\to\PP^3$, $[x,y,z,t]\to[x^d, y^d, z^d,t^d]$.

(2) Let $T_d^{-1}$ be the Macaulay inverse system of $T_d$ and denote by $Y_d$ the rational 3-fold defined as the closure of the image of the rational map
$\varphi _{T^{-1}_d}\colon\PP^3 \dashrightarrow \PP^{\binom{3+d}{d}-\mu (T_d)-1}$. By Theorem \ref{tea}, $Y_d$ satisfies a Laplace equation of order $d-1$.
\end{rem}

Our main goal will be to prove that the homogeneous ideal $I(X_d)$ of a GT-variety $X_d \subset \PP^{\mu(T_d)-1}$
is generated by quadrics if $d$ is even  and by quadrics and cubics if $d$ is odd (see Corollary  \ref{MAIN2})

%%%%%%%%%%%%%%%%%%%%%%%%%%%%%%%%%%%%%%%%%%%%%%%%%%

\section{The lattice of a GT-variety}
\label{GT-lattices}

As in the previous section, we fix  $d\geq 4$ and we write $d = 2k+\varepsilon = 3k'+\rho$, with $\varepsilon \in \{0,1\}$ and $\rho\in \{0,1,2\}$.
We want to determine the homogeneous ideal $I(X_{d})$ of the $GT$-threefold $X_{d} \subset \PP^{\mu(T_{d})-1}$ defined by the GT-system $T_d$. Since $X_d$ is an irreducible toric variety, $I(X_d)$ is a binomial ideal of codimension $\mu(T_d)-4$ associated to a lattice $L_{\eta}$. As we already pointed out our main goal is to prove that $I(X_d)$ is generated by quadrics if $d$ is even and by quadrics and cubics if $d$ is odd (see Corollary \ref{MAIN2}) but first we will explicitly describe a $\ZZ$-basis of the lattice $L_{\eta}$ associated to $I(X_d)$ (see Theorem \ref{Thm:MainTheorem}).

The ideal $T_{d}$ is generated by the set
$\{x^{\delta + 2\gamma+ (1-r)d}y^{rd - 2\delta - 3\gamma}z^{\delta}t^{\gamma}\; | (r,\gamma,\delta) \in \W\} \subset K[x,y,z,t]$ (see Section 3.1). All these monomials  are
uniquely determined by a triple $(r,\gamma,\delta) \in \W$ and often we will denote $x^{\delta + 2\gamma+ (1-r)d}y^{rd - 2\delta - 3\gamma}z^{\delta}t^{\gamma}$ by $w_{(r,\gamma,\delta)}$. 

\begin{defn} \rm \label{Defi:IdealI} We define the binomial ideal $I_d = (w_{(r_{1},\gamma_{1},\delta_{1})}w_{(r_{2},\gamma_{2},\delta_{2})} - w_{(r_{3},\gamma_{3},\delta_{3})}
w_{(r_{4},\gamma_{4},\delta_{4})} \;|$ $\; r_{1}+r_{2} = r_{3}+r_{4}, \; \gamma_{1} + \gamma_{2} = \gamma_{3} + \gamma_{4}, \; \delta_{1} + \delta_{2} =
\delta_{3} + \delta_{4}) \subset k[w_{(r,\gamma,\delta)}]_{(r,\gamma,\delta) \in \W}$.
\end{defn}

Let us  illustrate the above definition with an example.

\begin{ex}\label{Ej:Ideal4} \rm We take  $d = 4, \; (k = 2, k' = 1, \varepsilon = 0, \rho = 1)$. We have (Example \ref{Ej:4GeneIdeal}):
 $$T_{4} = (x^4, y^4, xy^2z, x^2z^2, x^2yt, z^4, y z^2 t,y^2t^2, xzt^2, t^4)$$ and
$$\mathcal{W}_{4} = \{(0,0,0), (1,0,0), (1,0,1), (1,0,2), (1,1,0), (2,0,4), (2,1,2), (2,2,0),(2,2,1),(3,4,0)\}.$$
Solving the equation $(r_{1},\gamma_{1},\delta_{1}) + (r_{2}, \gamma_{2},\delta_{2}) = (r_{3},\gamma_{3},\delta_{3}) + (r_{4},\gamma_{4},\delta_{4})$ in $\mathcal{W}_{4}$ we obtain twelve generators for $I_4$:

$$\begin{array}{lcl}
w_{(0,0,0)}w_{(2,0,4)} -  w_{(1,0,2)}^{2} & &
w_{(0,0,0)}w_{(2,1,2)} -  w_{(1,0,2)}w_{(1,1,0)}\\
w_{(0,0,0)}w_{(2,2,0)} -  w_{(1,1,0)}^{2}& &
w_{(1,0,0)}w_{(1,0,2)} -  w_{(1,0,1)}^{2}\\
w_{(1,0,0)}w_{(2,2,1)} -  w_{(1,0,1)}w_{(2,2,0)}& &
w_{(1,0,0)}w_{(3,4,0)} -  w_{(2,2,0)}^{2}\\
w_{(1,0,1)}w_{(2,2,1)} -  w_{(1,0,2)}w_{(2,2,0)}& &
w_{(1,0,1)}w_{(3,4,0)} -  w_{(2,2,0)}w_{(2,2,1)}\\
w_{(1,0,2)}w_{(2,1,2)} -  w_{(1,1,0)}w_{(2,0,4)}& &
w_{(1,0,2)}w_{(2,2,0)} -  w_{(1,1,0)}w_{(2,1,2)}\\
w_{(1,0,2)}w_{(3,4,0)} -  w_{(2,2,1)}^{2}& &
w_{(2,0,4)}w_{(2,2,0)} -  w_{(2,1,2)}^{2}.
\end{array}$$
\end{ex}

By construction it follows that $I_{d}$ vanishes on $X_{d}$, and hence $I_{d} \subseteq I(X_{d})$. Let $k[w_{(r,\gamma,\delta)}^{\pm}]$ be the ring of Laurent polynomials over $k$. To each binomial  $$w_{(r_{1},\gamma_{1},\delta_{1})}w_{(r_{2},\gamma_{2},\delta_{2})} - w_{(r_{3},\gamma_{3},\delta_{3})}
w_{(r_{4},\gamma_{4},\delta_{4})} \in I_{d}$$ we associated a Laurent binomial $$w_{\alpha} :=
w_{(r_{1},\gamma_{1},\delta_{1})}w_{(r_{2},\gamma_{2},\delta_{2})}w_{(r_{3},\gamma_{3},\delta_{3})}^{-1}
w_{(r_{4},\gamma_{4},\delta_{4})}^{-1} - 1\in k[w_{(r,\gamma,\delta)}^{\pm}].$$
 They generate a  Laurent binomial ideal whose associated partial character is the trivial one
$\eta: L_{\eta} \to k^{*}$, sending $\eta(m) = 1$ for all $m \in L_{\eta}$, where $L_{\eta} = \langle \alpha \;|\; w^{\alpha_{+}}- w^{\alpha_{-}}
\in I_{d}\rangle$. In turn, the partial character $\eta$ induces a lattice ideal $I_{+}(\eta) = (w^{\alpha_{+}}- w^{\alpha_{-}} \in k[w_{(r,\gamma,\delta)}]\;|\; \alpha \in L_{\eta})$.

Now we state the main result of this section.
\begin{thm}\label{Thm:MainTheorem} \begin{enumerate}
\item The lattice $L_{\eta}$ is saturated and $rk(L_{\eta}) = \mu(T_{d})-4$.
\item $I_{+}(\eta) = (\prod_{i=1}^{n}w_{(r_{i},\gamma_{i},\delta_{i})}
-\prod_{i=1}^{n}w_{(r_{i}',\gamma_{i}',\delta_{i}')} \in k[w_{(r,\gamma,\delta)}] \;\mid\; \sum_{i=1}^{n}r_{i} = \sum_{i=1}^{n}r_{i}', \;
\sum_{i=1}^{n}\gamma_{i} = \sum_{i=1}^{n}\gamma_{i}', \;
\sum_{i=1}^{n}\delta_{i} = \sum_{i=1}^{n}\delta_{i}'\}$.
\end{enumerate}
\end{thm}

\begin{cor}\label{igualtat_ideals} $I(X_{d}) = I_{+}(\eta)$.
\end{cor}

\begin{proof} Theorem \ref{Thm:MainTheorem} (1) implies that $I_{+}(\eta)$ is a prime ideal of codimension $4$ (see \cite[Corollary 2.5 and 2.6]{ES}).
From Theorem \ref{Thm:MainTheorem} (2) it follows that $I_{+}(\eta)$ vanishes in $X_{d}$, i.e. $I_{+}(\eta) \subset I(X_d)$. Therefore, $I_{+}(\eta)$ is the homogeneous ideal of an irreducible 3-dimensional variety contained in $X_d$. Since $X_d$ is irreducible we conclude that  $I_{+}(\eta) = I(X_d)$ which proves what we want.
\end{proof}

We trivially have $I_d\subset  I_{+}(\eta)=I(X_d)$. In next section we will discuss whether the equality holds. Now we devote the rest of this section to prove Theorem \ref{Thm:MainTheorem}.

\begin{defn}\rm
Fixed $n \geq 2$, we define a {\em suitable $n$-binomial} to be a nonzero  binomial $w^{\alpha} = w^{\alpha_{+}} - w^{\alpha_{-}} = \prod_{i=1}^{n}
w_{(r_{i},\gamma_{i},\delta_{i})} - \prod_{i=1}^{n} w_{(r_{i}',\gamma_{i}',\delta_{i}')}$ satisfying $\sum_{i=1}^{n}r_{i} = \sum_{i=1}^{n}r_{i}'$, $\sum_{i=1}^{n}
\gamma_{i} = \sum_{i=1}^{n}\gamma_{i}'$ and $\sum_{i=1}^{n}\delta_{i} = \sum_{i=1}^{n}
\delta_{i}'$.
\end{defn}

\begin{rem} \rm Any suitable $n$-binomial $w^{\alpha}$ vanishes in $X_{d}$. Therefore,  all suitable $n$-binomials belong to $I(X_{d})$. Moreover, the generators of
$I_{d}$ are suitable $2$-binomials.
\end{rem}

\begin{defn} \rm
Given a suitable $n$-binomial $w^{\alpha} = w^{\alpha_{+}}-w^{\alpha_{-}}$, we note $supp_{+}(w^{\alpha})$ (respectively $supp_{-}(w^{\alpha})$) the support of the monomial $w^{\alpha_{+}}$ (respectively
support of $w^{\alpha_{-}}$). We say that $w^{\alpha}$ is {\em non trivial} if $supp_{+}(w^{\alpha}) \cap supp_{-}(w^{\alpha}) = \emptyset$. Otherwise, we say that $w^{\alpha}$ is {\em trivial}.
\end{defn}
\begin{ex}\rm The set of generators for $I_{4}$ in Example \ref{Ej:Ideal4} are the set of all non-trivial suitable $2$-binomials.
\end{ex}

\begin{defn} \rm \label{Defi:MonAdmits} Let $m = \prod_{i=1}^{n}w_{(r_{i},\gamma_{i},\delta_{i})} \in k[w_{(r,\gamma,\delta)}]$ be a monomial of degree $n$. We say that $m$ {\em admits a suitable  $n$-binomial} if there exists a monomial $m'=\prod_{i=1}^{n}w_{(r_{i}',\gamma_{i}',\delta_{i}')}\in k[w_{(r,\gamma,\delta)}]$ of degree $n$ such that $m-m'$ is a non trivial suitable $n$-binomial.
\end{defn}

Let us order the elements $(r,\gamma,\delta) \in \W$ lexicographically.

\begin{defn} \rm We say that $w_{(r,\gamma,\delta)} \in k[w_{(r,\gamma,\delta)}]$ admits a {\em special} $n$-binomial if there exists a non trivial suitable $n$-binomial $m-m' \in I_{+}(\eta)$ such that $(r,\gamma,\delta)= min \{supp(m-m')\}$.
\end{defn}

\begin{ex} \rm The element $w_{(0,0,0)} \in k[w_{(r,\gamma,\delta)}]$ admits a special $2$-binomial. Indeed, $w_{(0,0,0)}w_{(2,2k',0)}-w_{(1,k',0)}w_{(1,k',0)}$ is a non trivial suitable $2$-binomial and $(0,0,0)=min\{(0,0,0),$ $(2,2k',0),(1,k',0)\}$. While clearly the element $w_{(3,d,0)}$ does not admit a special $n$-binomial for any $n \geq 2$.
\end{ex}

\begin{ex} \rm For $d=4$, the set of elements admitting a special  $2$-binomial is ${\mathcal{W}_{4}}-\{(1,1,0),(2,1,2),(2,2,0),$ $(2,2,1),(3,4,0)\}$ while the element $(1,1,0)$ admits a special 3-binomial: $w_{(1,1,0)}w_{(2,1,2)}w_{(3,4,0)}-w_{(2,2,0)}w^2_{(2,2,1)}.$
\end{ex}

\begin{lem}  \label{Lemma:13=22} Each monomial $m = w_{(1,\gamma,\delta)}w_{(3,d,0)} \in k[w_{(r,\gamma,\delta)}]$ admits a special $2$-binomial except: $(\gamma ,\delta) =( k', \lfloor \frac{\rho}{2} \rfloor )$ if $\rho \neq 0$, and $\gamma = \delta = 0$ if $\varepsilon = 1$.
\end{lem}

\begin{proof} Fix $(1,\gamma,\delta) \in \W$. If there exists such monomial $m'$, it has to be of the form
$w_{(2,\gamma_{1},\delta_{1})}w_{(2,\gamma_{2},\delta_{2})}$ with
$0 \leq \gamma_{i} \leq 2k' + \lfloor \frac{2\rho}{3} \rfloor$, $
\max\{0,d-2\gamma_{i}\} \leq \delta_{i} \leq \lfloor \frac{2d-3\gamma_{i}}{2} \rfloor$,  $i = 1,2$, $\gamma + d=\gamma _1+\gamma_2$ and $\delta =\delta _1+\delta _2$. From this follows that when $\rho = 1$ and $\gamma =  k'$, there are no $\gamma_{1}$ and $\gamma_{2}$ summing $\gamma + d=4k'+1$. While for $\rho = 2$, we must have $\gamma_{1} = \gamma_{2} = 2k'+1$. But then $\delta_{1} = \delta_{2} = 0$, which cannot sum $\delta = 1$.

For the rest of $\gamma$'s, we set $\gamma_{1}:=  \lfloor \frac{d+\gamma}{2} \rfloor $ and $\gamma_{2} := \lceil \frac{d + \gamma}{2} \rceil$.
Observe that we always have $k \leq \gamma_{1},\gamma_{2} \leq 2k' + \lfloor \frac{\rho}{2} \rfloor$. From the properties of the floor and ceiling functions we have
$$\lfloor \frac{2d - 3\gamma_{1}}{2} \rfloor + \lfloor \frac{2d - 3\gamma_{2}}{2} \rfloor \leq \lfloor \frac{4d - 3(d + \gamma)}{2} \rfloor
= \lfloor \frac{d - 3\gamma}{2} \rfloor,$$
where the equality holds when $\gamma_{1}$ and $\gamma_{2}$ are not both odd. If the equality holds  we can  find values $\delta_{1}$ and $\delta_{2}$ such that $\delta_{1} + \delta_{2} = \delta$, as long as
$\delta \geq \max\{0,d-2\gamma_{1}\} + \max\{0,d-2\gamma_{2}\}$. The last condition always happens except for $\gamma = \delta = 0$ when $\varepsilon = 1$.

Finally, if $\gamma_{1}$ and $\gamma_{2}$ are odd (and, hence,
$\gamma \geq 2$),  the result follows taking $m' = w_{(2,\gamma_{1}+1, \lfloor \frac{2d-3(\gamma_{1}+1)}{2} \rfloor)}w_{(2,\gamma_{2}-1, \lfloor \frac{2d-3(\gamma_{2}-1)}{2} \rfloor)}$.
\end{proof}

\begin{prop}  All $w_{(1,\gamma,\delta)} \in k[w_{(r,\gamma,\delta)}]$  admit a special 2-binomial or 3-binomial.
\end{prop}

\begin{proof} It is enough to treat the 3 exemptions of Lemma \ref{Lemma:13=22}. For $\varepsilon =1$ and $(1,\gamma,\delta)=(1,0,0)$ it is enough to observe that
  $w_{(1,0,0)}w_{(2,2k',0)}-w_{(1,1,0)}w_{(2,2k'-1,0)}$ if $\rho = 0$, $w_{(1,0,0)}w_{(2,2k',1)}-w_{(1,0,1)}w_{(2,2k',0)}$ if $\rho = 1$ and
$w_{(1,0,0)}w_{(2,2k'+1,0)}-w_{(1,1,0)}w_{(2,2k',0)}$ if $\rho = 2$ are special $2$-binomials.

 For $(1,\gamma,\delta)=(1,k',\lfloor \frac{\rho}{2} \rfloor)$  and $\rho \neq 0$, the monomial  $w_{(1,k',\lfloor \frac{\rho}{2} \rfloor)}$ does not  admit a special $2$-binomial. However, $w_{(1,k',0)}w_{(2,2k'-1,2)}w_{(3,d,0)} - w_{(2,2k',0)}w_{(2,2k',1)}^{2}$ for $\rho =1$ and $w_{(1,k',1)}w_{(2,2k',1)}$ $w_{(3,d,0)}-w_{(2,2k',2)}w_{(2,2k'+1,0)}^{2}$ for $\rho =2$ are special  $3$-binomials.
\end{proof}

\begin{prop} All $w_{(2,\gamma,\delta)} \in k[w_{(r,\gamma,\delta)}]$ admit a special 2-binomial or 3-binomial except
$\{w_{(2,2k'-1,0)},w_{(2,2k'-1,1,)},w_{(2,2k',0)}\}$ if $\rho = 0$, $\{w_{(2,2k'-1,2)},w_{(2,2k',0)},w_{(2,2k',1)}\}$ if $\rho = 1$, and
$\{w_{(2,2k',1)},w_{(2,2k',2)},w_{(2,2k'+1,0)}\}$ if $\rho = 2$.
\end{prop}

\begin{proof}
For any $(2,\gamma,\delta)\in \W $ different from the excluded cases we consider the monomial $m = w_{(2,\gamma,\delta)}w_{(2, 2k'+\lfloor \frac{\rho}{2} \rfloor, \lceil \frac{\rho}{2} \rceil - \lfloor \frac{\rho}{2} \rfloor)}$. For convenience we note
$\gamma'= 2k'+\lfloor \frac{\rho}{2} \rfloor$ and $\delta' = \lceil \frac{\rho}{2} \rceil - \lfloor \frac{\rho}{2} \rfloor $. Set $\gamma_{1}:=
\gamma+1$ and $\gamma_{2}:= \gamma'-1$. Unless $\gamma$ and $\gamma'$ are even, and $\delta =(2d-3\gamma)/2$ (hence $\rho \ne 2$), there exists
$\delta _i$ with $\max\{0,d-2\gamma_{i}\} \leq \delta_{i}
\leq \lfloor \frac{2d-3\gamma_{i}}{2} \rfloor$ such that
$\delta_{1} + \delta_{2} = \delta + \delta'$.

If $\gamma$ and $\gamma'$ are even, $\delta =(2d-3\gamma)/2$ and $\gamma < 2k'-2$ we take $\gamma_{1}:= \gamma+2$ and $\gamma_{2} := 2k'-2$. Then, there exists
$\delta _i$ with $\max\{0,d-2\gamma_{i}\} \leq \delta_{i}
\leq \lfloor \frac{2d-3\gamma_{i}}{2} \rfloor$ such that
$\delta_{1} + \delta_{2} = \delta + \delta'$.

If $\gamma =2k'-2$ and $\rho =1$, $w_{(2,2k'-2,4)}w_{(2,2k',0)} - w_{(2,2k'-1,2)}^{2}$ is a special $2$-binomial when $\rho = 1$.
Finally, if $\rho = 0$, $\gamma = 2k'-2$ and $\delta =3$, the element $(2,2k'-2,3)$ does not admit a special $2$-binomial but it admits a special 3-binomial:
 $w_{(2,2k'-2,3)}w_{(2,2k'-1,0)}w_{(2,2k',0)}-w_{(2,2k'-1,1)}^{3}$.
\end{proof}

From now on we set:
\begin{itemize}
 \item  $\W' = \W - \{(2,2k'-1,0),(2,2k'-1,1),(2,2k',0),(3,d,0)\}$ if $\rho = 0$,
 \item  $\W' = \W - \{(2,2k'-1,2),(2,2k',0),(2,2k',1),(3,d,0)\}$ $\rho = 1$, and
 \item  $\W' = \W - \{(2,2k',0), (2,2k',1),(2,2k',2),(3,d,0)\}$ $\rho = 2$.
\end{itemize}

Up to now we have seen that for any $(r,\gamma,\delta) \in \W'$ the variable $w_ {(r,\gamma,\delta)}$ admits a special 2-binomial or 3-binomial.

For each $(r,\gamma,\delta) \in \W'$ set $D_{(r,\gamma,\delta)}$ to be one of its special binomials and note $\alpha_{(r,\gamma,\delta)}$ its lattice point. We call $\{D_{(r,\gamma,\delta)}\}_{(r,\gamma,\delta) \in \W'}$ a {\em system of special binomials} and $\{\alpha_{(r,\gamma,\delta)}\}_{(r,\gamma,\delta) \in \W'}$ its associated {\em system of lattice points.} The matrix associated to any system of lattice points is upper triangular. So we have the following result:

\begin{cor}\label{Col:Basis}  For any system of special binomials $\{D_{(r,\gamma,\delta)}\}_{(r,\gamma,\delta) \in \W'}$ its associated system of lattice points $\{\alpha_{(r,\gamma,\delta)}\}_{(r,\gamma,\delta) \in \W'}$ is a $\ZZ$-basis of $\ZZ^{\mu(T_{d})-4}$.
\end{cor}

\begin{ex}\label{Ex:System4} \rm For $d=4$ we can chose as a system of special binomials
$$\begin{array}{l}

D_{(0,0,0)}:= w_{(0,0,0)}w_{(2,2,0)} -  w_{(1,1,0)}^{2}\\
D_{(1,0,0)}:= w_{(1,0,0)}w_{(3,4,0)} -  w_{(2,2,0)}^{2}\\
D_{(1,0,1)}:= w_{(1,0,1)}w_{(3,4,0)} -  w_{(2,2,0)}w_{(2,2,1)}\\
D_{(1,0,2)}:= w_{(1,0,2)}w_{(3,4,0)} -  w_{(2,2,1)}^{2}\\
D_{(1,1,0)}:= w_{(1,1,0)}w_{(2,1,2)}w_{(3,4,0)} - w_{(2,2,0)}w_{(2,2,1)}^{2}\\
D_{(2,0,4)}:= w_{(2,0,4)}w_{(2,2,0)} -  w_{(2,1,2)}^{2}.
\end{array}$$

The matrix associated to its system of  lattice points is
$$\begin{pmatrix}
1 & 0 & 0 & 0 &-2 & 0 & 0 & 1  & 0  & 0 \\
0 & 1 & 0 & 0 & 0 & 0 & 0 & -2 & 0  & 1 \\
0 & 0 & 1 & 0 & 0 & 0 & 0 & -1 & -1 & 1 \\
0 & 0 & 0 & 1 & 0 & 0 & 0 &  0 & -2 & 1 \\
0 & 0 & 0 & 0 & 1 & 0 & 1 & -1 & -2 & 1\\
0 & 0 & 0 & 0 & 0 & 1 &-2 & 1  & 0  & 0
\end{pmatrix}$$
So, $\{\alpha_{(r,\gamma,\delta)}\}_{(r,\gamma,\delta) \in \mathcal{W}_{4}'}$ is $\ZZ$-basis of $\ZZ^6$.
\end{ex}

Next  we prove that any system of lattice points is a $\ZZ$-basis of
the lattice $L_{\eta}$. In the sequel we fix $\{D_{(r,\gamma,\delta)}\}_{(r,\gamma,\delta)\in \W'}$ and its associated system of lattice points $\{\alpha_{(r,\gamma,\delta)}\}_{(r,\gamma,\delta) \in \W'}$. Rephrasing, we want to demonstrate that $L_{\rho} = \langle \alpha_{(r,\gamma,\delta)} \rangle_{(r,\gamma,\delta)\in \W'}$.
The lattice $L_{\eta}$ is generated by all suitable $2$-binomials. Thus it is enough to express the lattice point of any non-trivial suitable  $2$-binomial as a linear combination of $\{\alpha_{(r,\gamma,\delta)}\}_{(r,\gamma,\delta)\in \W'}$. So we fix
a non-trivial suitable $2$-binomial $$w^{\alpha_{0}} = m-m'=w_{(r_{1},\gamma_{1},\delta_{1})}w_{(r_{2},\gamma_{2},\delta_{2})} - w_{(r_{3},\gamma_{3},\delta_{3})}w_{(r_{4},\gamma_{4},\delta_{4})}$$ with associated lattice point $$\alpha_{0} = \alpha_{0}^{+} - \alpha_{0}^{-} = (r_{1},\gamma_{1},\delta_{1}) + (r_{2},\gamma_{2},\delta_{2}) - (r_{3},\gamma_{3},\delta_{3}) - (r_{4},\gamma_{4},\delta_{4}) \notin \{\alpha_{(r,\gamma,\delta)}\}_{(r,\gamma,\delta)\in \W'}.$$

Set $\alpha_{1} := \alpha_{0} - \sum_{\W',\alpha_{0}}^{+} \alpha_{(r_{i},\gamma_{i},\delta_{i})} + \sum_{\W',\alpha_{0}}^{-} \alpha_{(r_{i},\gamma_{i},\delta_{i})}$, where the summing $\sum_{\W',\alpha_{0}}^{+}$ (respectively, $\sum_{\W',\alpha_{0}}^{-}$) means that we only consider those elements $(r_{i},\gamma_{i},\delta_{i}) \in \W'\cap supp(\alpha_{0}^{+})$ (respectively, $\W' \cap supp(\alpha_{0}^{-})$). Therefore, $\alpha_{1}$ is a point of $L_{\eta}$  and its associated binomial $w^{\alpha_{1}}$ is a suitable $n$-binomial for some $n \geq 2$.
Furthermore, $supp(\alpha_{0}) \cap supp(\alpha_{1}) \cap \W' = \emptyset$ and all elements in $supp(\alpha_{1})$ are strictly bigger than $min\{supp(\alpha_{0})\}$.
If there exists a lattice point $(r,\gamma,\delta) \in \W' \cap supp(\alpha_{1})$, then we apply the same strategy to $\alpha_{1}$ and so on. Before continuing let us see how the procedure works by an example.

\begin{ex}\label{Ex:Reduc4} \em According to Example \ref{Ej:Ideal4} for $d=4$ we have 12 non-trivial suitable  2-binomials. Five of them are part of the system of special binomials that we fix in  Example \ref{Ex:System4}. Let us check that the seven remaining cases can be written as a linear combination of the system of special binomials fixed in Example \ref{Ex:System4}. The first step of the above induction process gives us:
$$\begin{array}{rclcl}
(0,0,0) + (2,0,4) & - &  2(1,0,2) & \to &\alpha_{1} = (2,0,4) + (2,2,0) - 2(2,1,2)\\
(0,0,0) + (2,1,2) & - & (1,0,2) - (1,1,0) & \to & \alpha_{1} = 0\\
(1,0,0) + (1,0,2) & - &  2(1,0,1) & \to &\alpha_{1} = 0\\
(1,0,0) + (2,2,1) & - &  (1,0,1) - (2,2,0) & \to &\alpha_{1} = 0\\
(1,0,1) + (2,2,1) & - &  (1,0,2) - (2,2,0) & \to &\alpha_{1} = 0\\
(1,0,2) + (2,1,2) & - &  (1,1,0) - (2,0,4) & \to &\alpha_{1} = -[(2,0,4) + (2,2,0) - 2(2,1,2)]\\
(1,0,2) + (2,2,0) & - & (1,1,0) - (2,1,2) & \to & \alpha_{1} = 0.\\
\end{array}$$
Since  $D_{(2,0,4)} = w_{(2,0,4)}w_{(2,2,0)} -w^2_{(2,1,2)}$ (see Example \ref{Ex:System4}), next step reduces $\alpha_{1}$ to $0$ in all cases.
\end{ex}

In general, this procedure defines inductively a sequence of lattice points
$\{\alpha_{1}, \hdots, \alpha_{s}, \hdots\}$ $\subseteq  L_{\eta}$, such that at any step $s$ of the induction process $supp(\alpha_{s-1}) \cap supp(\alpha_{s}) \cap \W' = \emptyset$ and $min\{supp(\alpha_{s})\}$ is strictly smaller than any element in the support of $\alpha_{s-1}$. So clearly this process stops, indeed $\W'$ is finite. Once it ends, we obtain a linear combination of $\{\alpha_{(r,\gamma,\delta)}\}_{(r,\gamma,\delta)\in \W'}\cup \{\alpha_{0}\}$, we denote it  $\alpha_{h} \in L_{\eta}$ for some $h \geq 1$. To achieve our goal it suffices to check that $\alpha_{h} = 0$. We note the elements of $\W-\W'$ by $l_{1},l_{2},l_{3}$ and $l_{4}$, ordered in the natural way. It is a matter of fact that $w^{\alpha_{h}}$ is a suitable $n$-binomial and  $supp(\alpha_{h}) \subseteq \W - \W'$. Thus there exist non negative integers $A_{1}, \hdots, A_{4}$ and $B_{1}, \hdots, B_{4}$ such that $\alpha_{h} = \sum_{i=1}^{4} A_{i}l_{i} - \sum_{i=1}^{4}B_{i}l_{i}$.

\begin{lem}  With the above notation, $(3,d,0) \notin supp(\alpha_{h})$.
\end{lem}

\begin{proof} Since $w^{\alpha_{h}}$ is a suitable $n$-binomial, it holds that $2(A_{1}+A_{2}+A_{3})+3A_{4} = 2(B_{1}+B_{2}+B_{3})+ 3B_{4}$. In other words, the $r$'s involved in $supp(\alpha_{h}^{+})$ and  $supp(\alpha_{h}^{-})$ form a two full partitions of the same length $A_{1}+A_{2}+A_{3}+A_{4} = B_{1}+B_{2}+B_{3}+B_{4}$ and weight $2(A_{1}+A_{2}+A_{3})+3(A_{4})$. So necessarily $A_{4} = B_{4}$ which proves what we want.
\end{proof}

For sake of completeness we specify $\alpha_{h}$ in each case.
\begin{itemize}
 \item $\alpha_{h} = (A_{1}-B_{1})(2,2k'-1,0)+(A_{2}-B_{2})(2,2k'-1,1)+(A_{3}-B_{3})(2,2k',0)$ when $\rho = 0$;
 \item $\alpha_{h} = (A_{1}-B_{1})(2,2k'-1,2)+(A_{2}-B_{2})(2,2k',0)+(A_{3}-B_{3})(2,2k',1)$ if $\rho = 1$; and
 \item $\alpha_{h} = (A_{1}-B_{1})(2,2k'+1,0)+(A_{2}-B_{2})(2,2k',1),+(A_{3}-B_{3})(2,2k',2)$ for $\rho = 2$.
\end{itemize}

 Since $w^{\alpha_{h}}$ is a suitable $n$-binomial, a straightforward computation shows that $A_{i} = B_{i}$, $i = 1,2,3$.
 $\square $

  %%%%%%%%%%%%%%%%%%%%%%%%%%%%%%%%%%%%%%%%%%%%%

\section{A minimal set of generators for GT-lattice ideals}
\label{minmalsetgenerators}

In the previous section we have stated that $I(X_{d})$ is a lattice ideal and we have given a $\ZZ$-basis of the associated lattice $L_{\eta}$ as well as a
system of generators of $I(X_{d})$ ( Theorem \ref{Thm:MainTheorem} (2)). Precisely, $I(X_{d})$ is generated by all non trivial suitable $n$-binomials with $n \geq 2$. Now we want to determine a minimal set of generators for $I(X_{d})$. More concretely,
we will prove that  the $GT$-lattice ideal $I(X_{d})$
is generated by quadrics if $d$ is even  and by  quadrics and cubics if $d$ is odd (Corollary \ref{MAIN2}).
As in previous sections  $d\geq 4$ and we write $d = 2k+\varepsilon = 3k'+\rho$, with $\varepsilon \in \{0,1\}$ and $\rho\in \{0,1,2\}$.

For each $n \geq 2$ we denote $I_{+}(\eta)_{n}$ the set of all suitable $n$-binomials and
$(I_{+}(\eta)_{n})$ the ideal of $k[w_{(r,\gamma,\delta)}]$ generated by them. Therefore, we have

\begin{equation}\label{igualtat_ideals1}  I(X_{d}) = \sum_{n \geq 2} (I_{+}(\eta)_{n}).
\end{equation}

\begin{defn}\label{Defi:In-sequence} \rm Let $w^{\alpha} = w^{\alpha_{+}}- w^{\alpha_{-}}$ be a non trivial suitable $n$-binomial. By an $I_+(\eta )_{n}$-{\em sequence} from $w^{\alpha_{+}}$ to $w^{\alpha_{-}}$ we mean a finite sequence $\{w^{a_{1}}, \hdots, w^{a_{t}}\}$ of monomials in $k[w_{(r,\gamma , \delta)}]$  satisfying the following two conditions:
\begin{itemize}
\item[(i)] $w^{a_{1}} = w^{\alpha_{+}}, w^{a_{t}} = w^{\alpha_{-}}$ and
\item[(ii)] For all $1\le j<t$,  $w^{a_{j}}-w^{a_{j+1}}$ is  a trivial suitable $n$-binomial.
\end{itemize}
\end{defn}

The second condition in the above definition says that for each $1 \leq j < t$, there exists a variable $w_{(r_{j},\gamma_{j},\delta_{j})} \in supp(w^{a_{j}}) \cap supp(w^{a_{j+1}})$. Thus each $w^{a_{j}} - w^{a_{j+1}}$ belongs to $(I_{+}(\eta)_{n-1})$.

\begin{ex} \rm Any trivial suitable $n$-binomial $w^{\alpha^{+}} - w^{\alpha^{-}}$ gives rise to the $I_{+}(\eta)_{n}$-sequence $\{w^{\alpha_{+}}, w^{\alpha_{-}}\}$.
\end{ex}

\begin{ex} \rm Consider $d = 4$ and $I_{4}$ from Example \ref{Ej:Ideal4}. The lattice ideal $I_{4}$ is generated by all suitable $2$-binomials. Let us give some examples of $I_{+}(\eta)_{3}$-sequence. Set $w^{a_{1}} = w_{(0,0,0)}w_{(1,0,2)}w_{(2,1,2)}$. Since $w_{(1,0,2)}w_{(2,1,2)}-w_{(1,1,0)}w_{(2,0,4)}$ is a suitable $2$-binomial, $\{w^{a_{1}},w^{a_{2}}\}$  with  $w^{a_{2}} := w_{(0,0,0)}w_{(1,1,0)}w_{(2,0,4)}$ is an $I_{+}(\eta)_{3}$-sequence. Now observe that $w_{(0,0,0)}w_{(2,0,4)}
- w_{(1,0,2)}w_{(1,0,2)}$ is also a suitable $2$-binomial.
 Hence $w^{a_{2}}-w^{a_{3}}$  with  $w^{a_{3}} := w_{(1,1,0)} w_{(1,0,2)}w_{(1,0,2)}$ is trivial and so $\{w^{\alpha_{1}},
w^{a_{2}}, w^{a_{3}}\}$ is an $I_{+}(\eta)_{3}$-sequence from $w^{\alpha_{1}}$ to $w^{\alpha_{3}}$.

 As another example of $I_{+}(\eta)_{3}$-sequence we have
$$\{w_{(1,0,0)}w_{(1,0,2)}w_{(2,2,1)}, w_{(1,0,1)}^{2}w_{(2,2,1)}, w_{(1,0,1)}w_{(1,0,2)}w_{(2,2,0)}, w_{(1,0,1)}w_{(1,1,0)}w_{(2,1,2)}\}$$
and the equality
$$\begin{array}{ll}
w_{(1,0,0)}w_{(1,0,2)}w_{(2,2,1)} - w_{(1,0,1)}w_{(1,1,0)}w_{(2,1,2)} = &
w_{(2,2,1)}w_{(1,0,0)}w_{(1,0,2)} - w_{(2,2,1)} w_{(1,0,1)}^{2} \\
& +  w_{(1,0,1)}^{2}w_{(2,2,1)} - w_{(1,0,1)}w_{(1,0,2)}w_{(2,2,0)} \\
& + w_{(1,0,1)}w_{(1,0,2)}w_{(2,2,0)} - w_{(1,0,1)}w_{(1,1,0)}w_{(2,1,2)}
\end{array}$$
shows that the non trivial $3$-binomial $w_{(1,0,0)}w_{(1,0,2)}w_{(2,2,1)} - w_{(1,0,1)}w_{(1,1,0)}w_{(2,1,2)}\in (I_{+}(\eta)_{2})$.  
\end{ex}

This last example illustrates very well what happens in general. Indeed, we have:

\begin{prop}\label{Prop:Criterion}  Fix $n \geq 3$ and let $w^{\alpha} = w^{\alpha_{+}}-w^{\alpha_{-}}$ be a suitable $n$-binomial. Then $w^{\alpha} \in (I_{+}(\eta)_{n-1})$ if and only if there exists an $I_{+}(\eta)_{n}$-sequence from $w^{\alpha_{+}}$ to $w^{\alpha_{-}}$.
\end{prop}

\begin{proof} Suppose that $w^{\alpha} \in (I_{+}(\eta)_{n-1})$. We note $I_{+}(\eta)_{n-1}$ $:=\{q_{1},\hdots, q_{N}\}$ with $N$ the number of all suitable $(n-1)$-binomials and $q_{j} = q_{j}^{u^{j}_{+}} - q_{j}^{u^{j}_{-}}$. By hypothesis there exist homogeneous linear forms $l_{1},\hdots, l_{N}$ such that $w^{\alpha_+} = l_{1}q_{1} + \cdots + l_{N}q_{N}+w^{\alpha_{-}}$. Now we write $l_{j} = a_{(0,0,0)}^{j}w_{(0,0,0)} + \cdots + a_{(3,d,0)}^{j}w_{(3,d,0)}$, where $a_{(r,\gamma,\delta)}^{j} \in k$ for all $(r,\gamma,\delta) \in \W$ and $j = 1,\hdots, N$. Therefore $w^{\alpha_{+}} = \sum_{j}^N\sum_{(r,\gamma,\delta) \in \W} (a_{(r,\gamma,\delta)}^{j}w_{(r,\gamma,\delta)}q_{j}^{u^{j}_{+}}
-  a_{(r,\gamma,\delta)}^{j}w_{(r,\gamma,\delta)}q_{j}^{u^{j}_{-}}) + w^{\alpha_{-}}
$. Hence, there exists $j_0$ such that $a^{j_0}_{(r_0,\gamma _0, \delta _0)}=1$ and
 $w^{\alpha_{+}} = w_{(r_0,\gamma _0,\delta _0)}q_{j_0}^{u^{j_0}_{+}}$
 or $a^{j_0}_{(r_0,\gamma _0, \delta _0)}=-1$ and
 $w^{\alpha_{+}} = w_{(r_0,\gamma _0,\delta _0)}q_{j_0}^{u^{j_0}_{-}}$ . Assume $a^{j_0}_{(r_0,\gamma_0, \delta_0)}=1$  (analogously we deal with the case $a^{j_0}_{(r_0,\gamma_0, \delta_0)}=-1$). Set
  $w^{a_{2}} = w_{(r_0,\gamma _0,\delta _0)}q_{j_0}^{u_{-}^{j_0}}$. We have
  $$ w^{\alpha_{+}} =w^{\alpha_{+}} - w^{a_{2}}+\sum _{(j,(r,\gamma, \delta))\ne (j_0,(r_0,\gamma _0, \delta _0))}(a_{(r,\gamma,\delta)}^{j}w_{(r,\gamma,\delta)}q_{j}^{u^{j}_{+}}
-  a_{(r,\gamma,\delta)}^{j}w_{(r,\gamma,\delta)}q_{j}^{u^{j}_{-}}) + w^{\alpha_{-}}.$$
  Thus
  $$w^{a_{2}}=\sum _{(j,(r,\gamma, \delta))\ne (j_0,(r_0,\gamma _0, \delta _0))}(a_{(r,\gamma,\delta)}^{j}w_{(r,\gamma,\delta)}q_{j}^{u^{j}_{+}}
-  a_{(r,\gamma,\delta)}^{j}w_{(r,\gamma,\delta)}q_{j}^{u^{j}_{-}}) + w^{\alpha_{-}}.$$
We iterate the process, first with $w^{a_{2}}$,  we construct the $I_{+}(\eta)_{n}$-sequence;   and taking into account that the number of summands decreases at each step we can assure that we end with what
we are looking for. We only have to note that the described process stops, since at each step we reduce the number of members of the linear combination, which is finite.
Therefore $w^{a_{t}}=w^{\alpha_{-}}$ for some $t>2$.
\end{proof}

Let $m$ be the smallest integer $m \geq 2$ such that
any suitable $(m+1)$-binomial of $I_{+}(\eta)_{m+1}$ admits a $I_{+}(\eta)_{(m+1)}$-sequence. By (\ref{igualtat_ideals}) and Proposition \ref{Prop:Criterion} we have
\begin{equation} \label{igualtat_ideals2}
I(X_d)=I_{+}(\eta) =\sum_{n \geq 2} (I_{+}(\eta)_{n})= \sum_{i=2}^{m} (I_{+}(\eta)_{i})
\end{equation}

\begin{notation} \rm For any odd integer $d\ge 5$, we define

\noindent $\mathcal{M}_{3}^{0} := \{w_{(0,0,0)}w_{(2,0,d)}w_{(1,0,\delta)} \}_{\delta = 0}^{k-1}\; \cup \, \{ w_{(1,0,0)}w_{(2,\gamma,d-2\gamma)}w_{(3,d,0)}\} _{\gamma =0}^{k-1} $ and

\noindent $\mathcal{M}_{3}^1= \mathcal{M}_{3}^2:= \{w_{(0,0,0)}w_{(2,0,d)}w_{(1,0,\delta)} \}_{\delta = 0}^{k-1}\; \cup \, \{ w_{(1,0,0)}w_{(2,\gamma,d-2\gamma)}w_{(3,d,0)}\} _{\gamma =0}^{k-1} \; \cup
\{w_{(0,0,0)}w_{(2,0,d)}w_{(3,d,0)}, $ $w_{(0,0,0)}w_{(1,0,0)}w_{(3,d,0)}\}.$
\end{notation}
Now we state our main result.

\begin{thm}\label{Teorema:BigTheorem}
\begin{itemize}
\item [(i)] If $d$ is even, for any $n \geq 3$ and any suitable $n$-binomial $w^{\alpha} = w^{\alpha_{+}}-w^{\alpha_{-}}$ there exists a $I_{+}(\eta)_{n}$-sequence from $w^{\alpha_{+}}$ to
$w^{\alpha_{-}}$.
\item [(ii)] If $d$ is odd, for any $n \geq 4$ and any suitable $n$-binomial $w^{\alpha} = w^{\alpha_{+}}-w^{\alpha_{-}}$ there exists a $I_{+}(\eta)_{n}$-sequence from $w^{\alpha_{+}}$ to
$w^{\alpha_{-}}$.
\item[(iii)] If $d$ is odd and  $n = 3$ then a suitable $3$-binomial $w^{\alpha} = w^{\alpha_{+}}-w^{\alpha_{-}}$ admits a $I_{+}(\eta)_{3}$-sequence from $w^{\alpha_{+}}$ to
$w^{\alpha_{-}}$ if and only if neither $w^{\alpha_{+}}$ nor $w^{\alpha_{-}}$  belong to $\mathcal{M}_{3}^{\rho} $.
\end{itemize}
\end{thm}

\begin{cor}\label{MAIN2} (1) If $d\ge 4$ is even, then  $I_{+}(\eta) = (I_{+}(\eta)_2) =I_{d}$.

(2) If $d\ge 5$ is odd, then $I_{+}(\eta) = (I_{+}(\eta)_2)+(I_{+}(\eta)_3)  =I_{d}+ (w^{\alpha} \in I_{+}(\eta)_{3} \; \mid \;  w^{\alpha_{+}} \in \mathcal{M}_{3}^{\rho } \quad \text{or} \quad w^{\alpha_{-}} \in \mathcal{M}_{3})^{\rho}$.
\end{cor}

We devote the rest of this section to prove Theorem \ref{Teorema:BigTheorem} but first  let us illustrate it with a couple of examples.

\begin{ex} \rm Using the software Macaulay2, we check that $I(X_{4}) = T_{4}$ (see Example
\ref{Ej:Ideal4}).
\end{ex}

\begin{ex} \label{Ex:Ideal5} \rm Fix $d = 5$, the binomial ideal $I_{5}$ is generated by twenty suitable $2$-binomials, all lattice points satisfying the equation $(r_{1},\gamma_{1},\delta_{1})+(r_{2}, \gamma_{2},\delta_{2}) = (r_{3},\gamma_{3},\delta_{3}) + (r_{4},\gamma_{4},\delta_{4})$.
$$\begin{array}{lll}
w_{(0,0,0)}w_{(2,1,3)}  -  w_{(1,0,2)}w_{(1,1,1)} & &
w_{(0,0,0)}w_{(2,2,1)}  -  w_{(1,1,0)}w_{(1,1,1)}\\
w_{(0,0,0)}w_{(2,2,2)}  -  w_{(1,1,1)}^{2} & &
w_{(1,0,0)}w_{(1,0,2)}  -  w_{(1,0,1)}^{2}\\
w_{(1,0,0)}w_{(1,1,1)}  -  w_{(1,0,1)}w_{(1,1,0)} & &
w_{(1,0,0)}w_{(2,2,2)}  -  w_{(1,0,1)}w_{(2,2,1)}\\
w_{(1,0,1)}w_{(1,1,1)}  -  w_{(1,0,2)}w_{(1,1,0)}& &
w_{(1,0,1)}w_{(2,2,2)}  -  w_{(1,1,0)}w_{(2,1,3)}\\
w_{(1,0,1)}w_{(2,2,2)}  -  w_{(1,0,2)}w_{(2,2,1)}& &
w_{(1,0,1)}w_{(2,3,0)}  -  w_{(1,1,0)}w_{(2,2,1)}\\
w_{(1,0,1)}w_{(3,5,0)}  -  w_{(2,2,1)}w_{(2,3,0)}& &
w_{(1,0,2)}w_{(2,1,3)}  -  w_{(1,1,0)}w_{(2,0,5)}\\
w_{(1,0,2)}w_{(2,2,2)}  -  w_{(1,1,1)}w_{(2,1,3)}& &
w_{(1,0,2)}w_{(2,3,0)}  -  w_{(1,1,1)}w_{(2,2,1)}\\
w_{(1,0,2)}w_{(2,3,0)}  -  w_{(1,1,0)}w_{(2,2,2)}& &
w_{(1,0,2)}w_{(3,5,0)}  -  w_{(2,2,2)}w_{(2,3,0)}\\
w_{(1,1,0)}w_{(3,5,0)}  -  w_{(2,3,0)}^{2}& &
w_{(2,0,5)}w_{(2,2,1)}  -  w_{(2,1,3)}^{2}\\
w_{(2,0,5)}w_{(2,3,0)}  -  w_{(2,1,3)}w_{(2,2,2)}& &
w_{(2,1,3)}w_{(2,3,0)}  -  w_{(2,2,1)}w_{(2,2,2)}
\end{array}$$
plus  eight non trivial suitable $3$-binomials of $I_{+}(\eta)_{3}$:
$$\begin{array}{lll}
w_{(0,0,0)}w_{(1,0,0)}w_{(2,0,5)}  -  w_{(1,0,1)}w_{(1,0,2)}^{2} & &
w_{(0,0,0)}w_{(1,0,0)}w_{(2,3,0)}  -  w_{(1,1,0)}^{3}\\
w_{(0,0,0)}w_{(1,0,0)}w_{(3,5,0)}  -  w_{(1,1,0)}^{2}w_{(2,3,0)} & &
w_{(0,0,0)}w_{(1,0,1)}w_{(3,5,0)}  -  w_{(1,0,2)}^{3}\\
w_{(0,0,0)}w_{(2,0,5)}w_{(3,5,0)}  -  w_{(1,1,1)}w_{(2,2,2)}^{2} & &
w_{(1,0,0)}w_{(2,0,5)}w_{(3,5,0)}  -  w_{(2,1,3)}w_{(2,2,1)}^{2}\\
w_{(1,0,0)}w_{(2,1,3)}w_{(3,5,0)}  -  w_{(2,2,1)}^{3} & &
w_{(1,1,1)}w_{(2,0,5)}w_{(3,5,0)}  -  w_{(2,2,2)}^{3}.
\end{array}$$

None of these  eight non trivial suitable $3$-binomials admits an $I_{+}(\eta)_{3}$-sequence from $w^{\alpha_{+}}$ to
$w^{\alpha_{-}}$ . For instance,  consider the non trivial suitable $3$-binomial $w^{\alpha } = w^{\alpha_{+}}-w^{\alpha_{-}} = w_{(0,0,0)}w_{(1,0,0)}w_{(2,0,5)}-w_{(1,0,1)}w_{(1,0,2)}^{2}$ of $I(X_{5})$. Assume that $\{w^{a_{1}},\hdots, w^{a_{t}}\}$ is an $I_{+}(\eta)_{3}$-sequence from
$w^{\alpha^{+}}$ to $w^{\alpha^{-}}$. Therefore $w^{\alpha^{+}} - w^{a_{2}}$ is a trivial
suitable $3$-binomial. So there are  $w_{(r,\gamma,\delta)} \in \{w_{(0,0,0)},w_{(1,0,0)},w_{(2,0,5)}\}$ and a non trivial suitable $2$-binomial $w^{\beta} = w^{\beta^{+}}- w^{\beta^{-}}$ such that $w^{\alpha^{+}}-w^{a_{t}} = w_{(r,\gamma,\delta)}w^{\beta}$ with $w^{\beta^{+}}$ or $w^{\beta^{-}}$ being
one of the monomials $w_{(0,0,0)}w_{(2,0,5)}$ or $w_{(1,0,0)}w_{(2,0,5)}$.
However all non trivial suitable $2$-binomials $w^{\xi}$ of $I_{5}$ verifies  $w^{\xi^{+}}, w^{\xi^{-}} \notin \{w_{(0,0,0)}w_{(1,0,0)}, w_{(0,0,0)}w_{(2,0,5)}, $ $w_{(1,0,0)}w_{(2,0,5)}\}$. Thus we conclude that the non trivial suitable $3$-binomial $w_{(0,0,0)}w_{(1,0,0)}w_{(2,0,5)}- w_{(1,0,1)}w_{(1,0,2)}^{2} \notin (I_{+}(\eta))_{2}) = I_{5}$ (see Proposition \ref{Prop:Criterion}).
\end{ex}

Now we develop our main techniques in constructing $I_{+}(\eta)_{n}$-sequences. Let $m = \prod_{i=1}^{n}w_{(r_{i},\gamma_{i},\delta_{i})}$ be a monomial of degree $n \geq 2$ and let  $w_{(r_{i_{j}},\gamma_{i_{j}}, \delta_{i_{j}})}$  be $f$ variables on the support of $m$, where $1 \leq f < n$. If $m_{f} = \prod_{j=1}^{f} w_{(r_{i_{j}},\gamma_{i_{j}}, \delta_{i_{j}})}$ admits a suitable $f$-binomial $m_{f}-m_{f}'$, then $m - m_{f}'\prod_{supp(m) - supp(m_{f}))}w_{(r_{i},\gamma_{i},\delta_{i})}$ is a trivial suitable $n$-binomial.
So determining whether a monomial admits a suitable $f$-binomial gives us a method to construct $I_{+}(\eta)_{n}$-sequence from a given monomial.
Let us start analyzing whether a monomial $w_{(r,\gamma ,\delta )}w_{(r',\gamma ',\delta ')}$ of degree 2 admits a suitable 2-monomial.

\begin{lem}\label{Lemma:02=11}  Any monomial $m = w_{(0,0,0)}w_{(2,\gamma,\delta)} \in k[w_{(r,\gamma,\delta)}]$ admits a special suitable $2$-binomial, with the following  exceptions: $(\gamma,\delta) = (2k'+\lfloor \frac{\rho}{2} \rfloor, \lceil \frac{\rho}{2} \rceil - \lfloor \frac{\rho}{2} \rfloor)$ if $\rho \neq 0$, and $\gamma = 0$ if $\varepsilon = 1$.
\end{lem}

\begin{proof}  If $m$ admits a suitable $2$-binomial $m-m'$ necessary $m'=w_{(1,\gamma_{1},\delta_{1})}w_{(1,\gamma_{2},\delta_{2})}$ with $0 \leq \gamma_{i} \leq k'$,
 $0 \leq \delta_{i} \leq \lfloor \frac{d-3\gamma_{i}}{2} \rfloor$ for $i = 1,2$, and  $\gamma_{1} + \gamma_{2} = \gamma$ and $\delta_{1} + \delta_{2} = \delta$. From this follows that $(2,\gamma,\delta)$ cannot be $(2,2k'+1,0)$ in case $\rho = 2$, $(2,2k',1)$ if $\rho = 1$ and $\gamma =0$ if $\varepsilon =1.$

 Otherwise we set $\gamma_{1}:= \lfloor \frac{\gamma}{2} \rfloor$ and $\gamma_{2} :=
\lceil \frac{\gamma}{2} \rceil$.
If $d$ is even and $\gamma _1, \gamma _2$ are odd or $d$ is odd and $\gamma_1, \gamma_2$ are even  we take $m' = w_{(1,\gamma_{1}+1,\lfloor
\frac{d-3(\gamma_{1}+1)}{2} \rfloor)} w_{(1,\gamma_{2}-1,\lfloor \frac{d-3(\gamma_{2}-1)}{2}
\rfloor)}$. In any other case we take $m' = w_{(1,\gamma_{1},\lfloor
\frac{d-3\gamma_{1}}{2} \rfloor)} w_{(1,\gamma_{2},\lfloor \frac{d-3\gamma_{2}}{2}
\rfloor)}$.
\end{proof}

\begin{lem}\label{Lema:100r2}  Suppose $\varepsilon = 1$.
\begin{itemize}
\item[(i)] Any monomial $m = w_{(1,0,0)}w_{(2,\gamma,\delta)}$ admits a suitable $2$-binomial except for $\gamma = 0, \hdots, k+1$ and $\delta = \max\{0,d-2\gamma\}$.
\item[(ii)]  Any monomial $m = w_{(1,\gamma,\delta)}w_{(2,0,d)}$ admits a suitable $2$-binomial except $\gamma = 0$ and $\delta  = 0,\hdots, k$ or  $\gamma = 1$ and $\delta = k-1$.
\end{itemize}
\end{lem}

\begin{proof} (i) We want to determine a monomial $m' = w_{(1,\gamma_{1},\delta_{1})}w_{(2,\gamma_{2},\delta_{2})}$ such that $m-m' \in I_{+}(\eta)_{2}$. If $\delta > max\{0,d-2\gamma\}$, we take
$(1,\gamma_{1},\delta_{1}) = (1,0,1)$ and $(2,\gamma_{2},\delta_{2})
= (2,\gamma,\delta-1)$. Let us to consider the remainder cases $(2,\gamma,max\{0,d-2\gamma\})$ with $\gamma = 0,\hdots,2k'+\lfloor \frac{\rho}{2} \rfloor$.
If $\gamma >k+1 $, $(2,\gamma,max\{0,d-2\gamma\})=(2, \gamma , 0)$ and we take $(1,\gamma_{1},\delta_{1}) = (1,1,0)$ and $(2,\gamma_{2},\delta_{2})
= (2,\gamma -1,0)$. For $0\le \gamma \le k+1$ a monomial $m'$ with
$ \gamma _1+\gamma_2=\gamma $ and $\delta _1+\delta _2=\delta $ does not exist because  we necessarily have $\gamma _1=i$ and $\gamma_2=\gamma - i$ for some $0\le i \le \gamma$, $0\le \delta _1\le \lfloor \frac{d-3i}{2}\rfloor $   and $d-2(\gamma -i)\le \delta _2\le \lfloor \frac{2d-3\gamma +3i}{2}\rfloor$ which give us $\delta <d-2(\gamma -i)\le \delta _1+\delta _2$.

The proof of (ii) is  analogous and we leave it to the reader.
 \end{proof}

\begin{rem} \rm \label{Remark:case03} (1)  The monomial $w_{(0,0,0)}w_{(3,d,0)}$ admits a non trivial suitable $2$-binomial only when $\rho = 0$. Indeed, assume that $w_{(0,0,0)}w_{(3,d,0)}-w_{(1,\gamma_{1},\delta_{1})}w_{(2,\gamma_{2},\delta_{2})}$ is a suitable $2$-binomial. Then we have
$\gamma_{1} + \gamma_{2} = 3k'+\rho = k' + 2k' + \rho$. So $\gamma_{1} = k'$ and $\gamma_{2} = 2k'+\rho = 2k' + \lfloor \frac{\rho}{2} \rfloor$. The last equality is achieved only when $\rho = 0$.

(2) Suppose $\rho = 1$. Any monomial $m = w_{(1,k',0)}w_{(2,\gamma,\delta)}$ admits a suitable $2$-binomial except when $\gamma = 2k'$.
Indeed,  if $\gamma <2k'$ we take $(r_1,\gamma_{1},\delta_{1}) = (1,k'-1,\delta _1)$ and $(r_2,\gamma_{2},\delta_{2})
= (2,\gamma +1,\delta_2)$ with $\delta =\delta_1+\delta_2$, $0\le \delta _1\le \lfloor\frac{d-3k'+3}{2}\rfloor$ and $max\{0,d-2\gamma -2\}\le \delta _2\le \lfloor\frac{2d-3\gamma-3}{2}\rfloor$.
If $\gamma =2k'$, since $\gamma _1<k'$ and $\gamma_2\le 2k'$ we will never have $\gamma=\gamma_1+\gamma_2$.

(3)   Suppose $\rho = 2$.  Clearly $w_{(1,k',1)}w_{(2,2k'+1,0)}$  and $w_{(1,k',1)}w_{(2,2k,2)}$ if $\varepsilon =0$ do not
 admit a suitable $2$-binomial.
 If $d-3\gamma$ is even and $\delta=\frac{2d-3\gamma}{2}$, we take
 $m' = w_{(1,k'-2,\lfloor
\frac{d-3(k'-2)}{2} \rfloor)} w_{(2,\gamma +2,\lfloor \frac{2d-3(\gamma +2)}{2}
\rfloor)}$. In any other case we take $m' = w_{(1,k'-1,\lfloor
\frac{d-3(k'-1)}{2} \rfloor)} w_{(2,\gamma +1,\lfloor \frac{2d-3(\gamma +1)}{2}
\rfloor)}$.
 Any monomial $m = w_{(1,k',1)}w_{(2,\gamma,\delta)}$ admits a suitable $2$-binomial except: $\gamma =2k'+1$ and $(\gamma, \delta )=( 2k', 2)$ when $\varepsilon = 0$.

(4) Suppose $\rho = 2$. Any monomial $m = w_{(1,\gamma,\delta)}
w_{(2,2k'+1,0)}$ admits a suitable $2$-binomial except $\gamma = k'$. The proof is analogous and we left it to the reader.
\end{rem}

\begin{prop}\label{Prop:Cubics1}   Suppose $\varepsilon = 1$. Let $w^{\alpha} = w^{\alpha_{+}}-w^{\alpha_{-}}$ be a non-trivial $3$-binomial. If $w^{\alpha_{+}}$ or $w^{\alpha_{-}}$ is one of the following:
\begin{itemize}
\item [(i)] $w_{(0,0,0)}w_{(2,0,d)}w_{(1,0,\delta)}, \;\; \delta = 0,\hdots, k;$
\item [(ii)] $w_{(0,0,0)}w_{(2,0,d)}w_{(3,d,0)}$ and $\rho \ne 0$;
\item [(iii)] $w_{(0,0,0)}w_{(1,0,0)}w_{(3,d,0)}$ and $\rho \ne 0$;
\item [(iv)] $w_{(1,0,0)}w_{(2,\gamma,d-2\gamma)}w_{(3,d,0)}, \;\; \gamma = 0,\hdots, k$ and $w_{(1,0,0)}w_{(2,k+1,0)}w_{(3,d,0)};$
\end{itemize}
then there is no an  $I_{+}(\eta)_{3}$-sequence from $w^{\alpha_{+}}$ to $w^{\alpha_{-}}$. In particular, $w^{\alpha} \notin (I_{+}(\eta)_{2}) = I_{d}$ and $I_d\varsubsetneq I(X_d)$
\end{prop}

\begin{proof} Let $\{w^{a_{1}}, \hdots, w^{a_{t}}\}$ be  an $I_{+}(\eta)_{3}$-sequence from $w^{\alpha_{+}}$ to $w^{\alpha_{-}}$.
 So there exist $w_{(r,\gamma,\delta)} \in k[w_{(r,\gamma,\delta)}]$ and a suitable $(n-1)$-binomial $w^{\alpha'}$ such that $w^{a_{1}}-w^{a_{2}}
= w_{(r,\gamma,\delta)}w^{\alpha'}$. This implies that we can find a monomial of degree $(n-1)$ on the support of $w^{\alpha_{+}}$ (respectively $w^{\alpha_{-}}$) admitting a suitable $(n-1)$-binomial.

May we suppose that $w^{\alpha_{+}}$ belongs to the above list. From Lemmas \ref{Lemma:02=11}, \ref{Lemma:13=22} and \ref{Lema:100r2}  it follows that any monomial of degree $2$ that we can form from
$supp(w^{u_{+}})$ in (i), (ii) and (iii) do not admit a non trivial suitable $2$-binomial contradicting the existence of an $I_{+}(\eta)_{3}$-sequence from $w^{\alpha_{+}}$ to $w^{\alpha_{-}}$.

 In case  (iv) we only have to treat the monomials associated to
$(1,0,0) + (2,\gamma,d - 2\gamma)$ for $\gamma = 0,\hdots, k$ and
$(1,0,0) + (2,k+1,0)$. Fix $\gamma \in \{0,\hdots, k+1\}$ and assume that there exist $(1,\gamma_{1},\delta_{1})> (1,0,0)$ and $(2,\gamma_{2},\delta_{2})<(2,\gamma,\delta)$ such that $\gamma_{1}+\gamma_{2} = \gamma$ and $\delta_{1} + \delta_{2} = d - 2\gamma$ for $\gamma = 0,\hdots, k$; and $\delta_{1} + \delta_{2} = 0$ for $\gamma = k+1$. Write $\gamma_{2} = \gamma-\gamma_{1}$, therefore $\delta_{2} \geq \delta + 2\gamma_{1}$. From this we deduce that $\delta_{1} + \delta_{2} \geq
\delta_{1} + \delta + 2\gamma_{1}$ and hence $\delta_{1} + 2\gamma_{1}$ must be zero, that is $\delta_{1} = 0 = \gamma_{1}$, which is a contradiction.  \end{proof}

\begin{prop}\label{Lemma:Cubics}  Suppose $\varepsilon = 1$.

(1) The monomials
\begin{itemize}
\item [(i)] $w_{(0,0,0)}w_{(2,0,d)}w_{(1,0,\delta)}, \;\; \delta = 0,\hdots, k-1;$
\item [(ii)] $w_{(0,0,0)}w_{(2,0,d)}w_{(3,d,0)};$
\item [(iii)] $w_{(0,0,0)}w_{(1,0,0)}w_{(3,d,0)};$
\item [(iv)] $w_{(1,0,0)}w_{(2,\gamma,d-2\gamma)}w_{(3,d,0)}, \;\; \gamma = 0,\hdots, k-1$
\end{itemize}
admit a suitable $3$-binomial of $I_{+}(\eta)_{3}$.

 (2) The monomials  $w_{(0,0,0)}w_{(2,0,d)}w_{(1,0,k)}$, $w_{(1,0,0)}w_{(2,k,1)}w_{(3,d,0)}$ and $w_{(1,0,0)}w_{(2,k+1,0)}w_{(3,d,0)}$ do not admit a suitable $3$-binomial.
\end{prop}

\begin{proof} (1) It is enough to exhibit explicitly a $3$-binomial in each case.
 \begin{itemize}
\item[(i)] For any $\delta \in \{0,\hdots, k-1\}$ we have  $w_{(0,0,0)}w_{(2,0,d)}w_{(1,0,\delta)} - w_{(1,0,k)}w_{(1,0,k)}w_{(1,0,\delta+1)}$ belong to $I_{+}(\eta)_{3}$.
\item[(ii)] We have $w_{(0,0,0)}w_{(2,0,d)}w_{(3,d,0)} - w_{(1,0,k)}w_{ (2,k,\lceil \frac{k+1}{2} \rceil)} w_{(2,k+1,\lfloor \frac{k+1}{2} \rfloor)} \in I_{+}(\eta)_{3}$.
\item[(iii)] We have $w_{(0,0,0)}w_{(1,0,0)}w_{(3,d,0)} - w_{(1,\lfloor
\frac{k'+ \lceil \frac{\rho}{2}\rceil}{2}\rfloor,0)}w_{(1, \lceil\frac{k'+ \lceil \frac{\rho}{2}\rceil}{2}\rceil,0)}w_{(2,2k' + \lfloor \frac{\rho}{2} \rfloor,0)} \in I_{+}(\eta)_{3}$.
\item[(iv)] For all $0 \leq \gamma \leq k-1$, $w_{(1,0,0)}w_{(2,\gamma,d-2\gamma)}w_{(3,d,0)} - w_{(2,\gamma+1,\max\{0,d-2\gamma-2\})}w_{ (2,k,1) + (2,k,1)}$ is a suitable 3-binomial.
\end{itemize}

(2) If $w_{(0,0,0)}w_{(2,0,d)}w_{(1,0,k)}-w_{(1,\gamma_{1},\delta_{1})}w_{(1,\gamma_{2},\delta_{2})}w_{(1,\gamma_{3},\delta_{3})}$ is a suitable 3-binomial, we must have $\gamma_{1} = \gamma_{2} = \gamma_{3} = 0$ and $\delta_{1} + \delta_{2} + \delta_{3} = 3k+1$. However, $\delta_{1},\delta_{2},\delta_{3} \leq k$.
 If $w_{(2,\gamma_{1},\delta_{1})}w_{(2,\gamma_{2},\delta_{2})}w_{(2,\gamma_{3},\delta_{3})}$ forms a $3$-binomial, in these cases, since $\delta_{1}+\delta_{2}+\delta_{3} \in \{0,1\}$, we must have  $\gamma_{1},\gamma_{2},\gamma_{3} \geq k$. But when $\gamma = k$, $\gamma_{1}+\gamma_{2}+\gamma_{3} = 3k+1$ implies that some $(2,\gamma_{i},\delta_{i}) = (2,k,1)$.  Finally, if $\gamma = k+1$, then $\gamma_{1},\gamma_{2},\gamma_{3} \geq k+1$ hence we find a similar argument.
 \end{proof}

Notice that the last two Propositions are false for even values of $d$. For instance we have that for $d$ even $\{w_{(0,0,0)}w_{(1,0,0)}w_{(2,0,d)}, w_{(1,0,k)}^{2}w_{(1,0,0)}\}$ is a $I_{+}(\eta)_{3}$-sequence. For sake of completeness we exhibit a complete example.

\begin{ex} \em We center in $I(X_{4}) = I_{4}$.
We only have to check that all monomials as in Proposition \ref{Lemma:Cubics}(2)  contain a
submonomial of degree $2$ admitting a non trivial suitable $2$-binomial. Indeed,
$w_{(0,0,0)}w_{(2,0,4)} - w_{(1,0,2)}^{2}$ and $w_{(1,0,0)}w_{(3,4,0)} - w_{(2,2,0)}^{2}$ are suitable $2$-binomials of $I_{4}$, from which the result follows.
\end{ex}

In the sequel we fix $n \geq 3$, otherwise indicated.
Any non trivial suitable $n$-binomial $w^{\alpha} = w^{\alpha_{+}} - w^{\alpha_{-}}$ is associated to a lattice point $\alpha$ of the form:
$$a(0,0,0) + \sum_{i=1}^{b} (1,\gamma_{i}^1,\delta_{i}^1) + \sum_{j=1}^{c} (2,\gamma_{j}^2,\delta_{j}^2) + e(3,d,0)$$
$$ - A(0,0,0)
 - \sum_{s=1}^{B} (1,\gamma_{s}^{1}, \delta_{s}^{1})
- \sum_{r=1}^{C} (2,\gamma_{r}^{2}, \delta_{r}^2) - E(3,d,0),$$
for integers $0 \leq a,b,c,e,A,B,C,E \leq n$ and $aA = 0 = eE$. Since $w^{\alpha}$ is a suitable $n$-binomial, we have restrictions $a + b + c + e = A+B+C+E$ and $b + 2c + 3e = B+2C+3E$.

\begin{prop} \label{Prop:Reductiontype12}
Let $w^{\alpha} = w^{\alpha_{+}}-w^{\alpha_{-}}$ be a non trivial suitable
$n$-binomial with $w_{(0,0,0)} \in supp(w^{\alpha})$ or $w_{(3,d,0)} \in supp(w^{\alpha})$. Assume that $w^{\alpha_{+}},w^{\alpha_{-}} \notin \mathcal{M}_{3}^{\rho}$. Then there exist $I_{+}(\eta)_{n}$-sequences $\{w^{\alpha_{+}}, \hdots, w^{\alpha_{+}'}\} $ and $\{w^{\alpha_{-}'}, \hdots, w^{\alpha_{-}}\}$ where $w_{(0,0,0)},w_{(3,d,0)} \notin supp(w^{\alpha_{+}'}) \cup supp(w^{\alpha_{-}'})$.
\end{prop}

\begin{proof} We write $w^{\alpha^{+}} = a(0,0,0) + \sum_{i=1}^{b} (1,\gamma_{i}^1,\delta_{i}^1) + \sum_{j=1}^{c} (2,\gamma_{j}^2,\delta_{j}^2) + e(3,d,0)$ and we assume that $a > 0$ or $e > 0$. Analogous we deal with $w^{\alpha_{-}}$. It is enough to see that we can always decrease the value of $a+e$ until we reach 0. We analyze separately several cases according to the  value of $d$:

\noindent \underline{Case 1:} Assume $\varepsilon =0$ and $\rho =0$. First we observe that the hypothesis $w^{\alpha }$ non-trivial implies $(b,c)\ne (0,0)$ or $(b,c)=(0,0)$ and $a=e$. If   $(b,c)=(0,0)$ and $a=e$ we have $w_{(0,0,0)}^aw_{(3,d,0)}^a=w_{(1,k',0)}^aw_{(2,2k',0)}^a$. Otherwise, since $m= w_{(3,d,0)}w_{(1,\gamma_{1}^1,\delta_{1}^1)}$
(resp. $m= w_{(0,0,0)}w_{(2,\gamma_{1}^2,\delta_{1}^2)}$)  admits a special  suitable $2$-binomial $m-m'$ with $m'=w_{(2,\gamma_{c+1}^2,\delta_{c+1}^2)}w_{(2,\gamma_{c+2}^2,\delta_{c+2}^2)}$ (resp. $m'=w_{(1,\gamma_{b+1}^1,\delta_{b+1}^1)}w_{(1,\gamma_{b+2}^1,\delta_{b+2}^1)}$),
 we can write
 $$w^{a_{1}} := w_{(0,0,0)}^{a}\prod_{i=2}^{b}w_{(1,\gamma_{i}^1,\delta_{i}^1)}\prod_{j=1}^{c+2}w_{(2,\gamma_{j}^2,\delta_{j}^2)}w_{(3,d,0)}^{e-1}$$

$$\text{(resp. } w^{a_{1}} := w_{(0,0,0)}^{a-1}\prod_{i=1}^{b+2}w_{(1,\gamma_{i}^1,\delta_{i}^1)}\prod_{j=2}^{c}w_{(2,\gamma_{j}^2,\delta_{j}^2)}w_{(3,d,0)}^{e}\text{ )}$$
\noindent and  build an $I_{+}(\eta)_{n}$-sequence $\{w^{\alpha_{+}},w^{a_1} \}$ with $deg _{w(0,0,0)}w^{a_1}+deg _{w(3,d,0)}w^{a_1}<a+e=deg _{w(0,0,0)}w^{\alpha _+}+deg _{w(3,d,0)}w^{\alpha _+}$ and we have decreased by 1 the value of $a+e$.

\noindent \underline{Case 2:} Assume $\varepsilon =0$ and $1\le \rho \le 2$. The hypothesis $w^{\alpha }$ non-trivial implies $(b,c)\ne (0,0)$ and we can argue as in Case 1 unless
$w^{\alpha_{+}}=w_{(0,0,0)}^aw_{(1,k',0)}^bw_{(2,2k',1)}^{c}w_{(3,d,0)}^e$ (resp. $w^{\alpha_{+}}=w_{(0,0,0)}^aw_{(1,k',1)}^bw_{(2,2k'+1,0)}^{c}w_{(3,d,0)}^e$) but such $w^{\alpha_{+}}$ does not admit a non-trivial $n$-binomial $w^{\alpha_{+}}-w^{\alpha_{-}}$.

\noindent \underline{Case 3:} Assume $\varepsilon =1$ and $ \rho =0$.
  Since $w_{(0,0,0)}w_{3,d,0)}=w_{(1,k',0)}w_{(2,2k',0)}$ we can argue as in the Case 1 unless $w^{\alpha_{+}}=w_{(0,0,0)}^aw_{(1,0,0)}^bw_{(2,0,d)}^c$ or $w^{\alpha_{+}}=w_{(1,0,0)}^bw_{(2,0,d)}^cw_{(3,d,0)}^e$, the fact that
$w^{\alpha_{+}}-w^{\alpha_{-}}$ is non-trivial implies $b,c>0$ and the hypothesis $w^{\alpha_{+}}\notin \mathcal{M}_{3}^{\rho}$ implies $a+b+c>3$ (resp.  $b+c+e>3$). Set $m=w_{(0,0,0)}w_{(1,0,0)}w_{(2,0,,d)}$ (resp. $m=w_{(1,0,0)}w_{(2,0,d)}w_{(3,d,0)}$). By Proposition \ref{Lemma:Cubics}
$w_{(0,0,0)}w_{(2,0,d)}w_{(1,0,0)} - w_{(1,0,k)}w_{(1,0,k)}w_{(1,0,1)}$
(resp. $w_{(1,0,0)}w_{(2,0,d)}w_{(3,d,0)} - w_{(2,1,d-2\})}w_{ (2,k,1) + (2,k,1)}$) and we apply the same game decreasing $a$ (resp. $e$) by one.

\noindent \underline{Case 4:} Assume $\varepsilon =1$ and $ 1 \leq \rho \leq 2$. Notice that from the hypothesis $w^{\alpha}$ non trivial we have $(b,c) \neq 0$. So we proceed as in Case 1 unless $w^{\alpha^{+}} = w_{(0,0,0)}^{a}w_{(1,0,0)}^{b}w_{(1,k',0)}^{c}w_{(2,0,d)}^{f}w_{(2,2k',1)}^{g}w_{(3,d,0)}^{e}$ with $(b,c,f,g) \neq (0,0,0,0)$ (resp. $w^{\alpha_{+}} =  w_{(0,0,0)}^{a}w_{(1,0,0)}^{b}w_{(1,k',1)}^{c}w_{(2,0,d)}^{f}w_{(2,2k'+1,0)}^{g}w_{(3,d,0)}^{e}$ and $(b,c,f,g) \neq (0,0,0,0)$). Since $w^{\alpha_{+}} \notin \mathcal{M}_{3}^{1}$ we have
$(c,g) \neq (0,0)$ or $a + b + f + g+ e > 3$.
By Proposition \ref{Lemma:Cubics} we have
$w_{(0,0,0)}w_{(1,0,0)}w_{(2,0,d)} - w_{(1,0,k)}w_{(1,0,k)}w_{(1,0,1)}$,
$w_{(0,0,0)}w_{(2,0,d)}w_{(3,d,0)} - w_{(1,0,k)}w_{(2,k,\lfloor \frac{k+1}{2} \rfloor)}w_{(2,k+1,\lfloor \frac{k+1}{2}\rfloor)}$,
$w_{(0,0,0)}w_{(1,0,0)}w_{(3,d,0)}$ $ - w_{(1,1,0)}w_{(1,k',0)}w_{(2,2k',0)}$
and
$w_{(1,0,0)}w_{(2,0,d)}$ $w_{(3,d,0)} - w_{(2,1,d-2)}w_{(2,k,1)}w_{(2,k,1)}$ are non trivial suitable $3$-binomials (resp. $w_{(0,0,0)}w_{(1,0,0)}w_{(2,0,d)} - w_{(1,0,k)}w_{(1,0,k)}w_{(1,0,2)}$,
$w_{(0,0,0)}w_{(2,0,d)}w_{(3,d,0)}$ $ - w_{(1,0,k)}w_{(2,k,\lfloor \frac{k+1}{2} \rfloor)}w_{(2,k+1,\lfloor \frac{k+1}{2}\rfloor)}$,
$w_{(0,0,0)}w_{(1,0,0)}w_{(3,d,0)} - w_{(1,1,0)}w_{(1,k',0)}w_{(2,2k'+1,0)}$ and
$w_{(1,0,0)}w_{(2,0,d)}$ $w_{(3,d,0)} - w_{(2,1,d-2)}w_{(2,k,1)}w_{(2,k,1)}$).
Then we argue as in Case 3 decreasing $a$ and $e$ unless $w^{\alpha_{+}} = w_{(0,0,0)}^{a}w_{(1,k',0)}^{c}w_{(2,2k',1)}^{g}w_{(3,d,0)}^{e}$ (resp. $w_{(0,0,0)}^{a}w_{(1,k',1)}^{c}$ $ w_{(2,2k'+1,0)}^{g}w_{(3,d,0)}^{e}$)
but such monomial does not admit a non trivial suitable $n$-binomial and the proof is completed.
\end{proof}

\begin{rem} \rm It is easy to observe that any suitable $n$-binomial $w^{\alpha} = w^{\alpha_{+}}-w^{\alpha_{-}} = \prod_{i=1}^{b}w_{(1,\gamma_{i}^{1},\delta_{i}^{1})}\prod_{j=1}^{c}w_{(2,\gamma_{j}^{2},\delta_{j}^{2})}
- \prod_{i=1}^{c'}w_{(1,\gamma_{i}^{3},\delta_{i}^{3})}\prod_{j=1}^{b'}w_{(2,\gamma_{j}^{4},\delta_{j}^{4})}$ satisfies $b = b'$ and $c = c'$.

\end{rem}

\begin{ex}\rm \label{Ex:elem4} (1) Fix $d=4$ and consider   the non trivial $3$-binomial
$$w_{(0,0,0)}w_{(1,0,0)}w_{(2,0,4)} - w_{(1,0,1)}w_{(1,0,1)}w_{(1,0,2)}.$$
Since $w_{(0,0,0)}w_{(2,0,4)} - w_{(1,0,2)}^{2}$ is a non trivial $2$-binomial, we define $w^{a_{1}}
= w_{(1,0,0)}w_{(1,0,2)}^{2}$ and we get
 an $I_{+}(\eta)_{3}$-sequence  $\{w_{(0,0,0)}w_{(1,0,0)}w_{(2,0,4)},w_{(1,0,0)}w_{(1,0,2)}^{2},w_{(1,0,1)}w_{(1,0,1)}w_{(1,0,2)}\}$ from $w_{(0,0,0)}w_{(1,0,0)}w_{(2,0,4)}$ to  $w_{(1,0,1)}w_{(1,0,1)}w_{(1,0,2)}$ where $w^{\alpha_{+}'} = w^{a_{1}}$.

\vskip 2mm (2) Fix $d=5$ and consider  $I(X_{5})$ and the non trivial $4$-binomial
$w_{(0,0,0)}w_{(1,0,0)}w_{(2,0,5)}w_{(3,5,0)}- w_{(1,1,0)}^{2}w_{(2,1,3)}w_{(2,2,2)}$.
We take the suitable $3$-binomial
 $w_{(0,0,0)}w_{(1,0,0)}w_{(2,0,5)}
- w_{(1,0,1)}w_{(1,0,2)}^{2}$ and we define
$w^{a_{1}} := w_{(1,0,1)}w_{(1,0,2)}^{2}w_{(3,5,0)}$. We observe that $w_{(0,0,0)}\notin supp(w^{a_1})$.
The monomial $w_{(1,0,1)}w_{(3,5,0)}$ admits a suitable $2$-binomial
$w_{(1,0,1)}w_{(3,5,0)}-w_{(2,2,1)}w_{(2,3,0)}$. We now define $w^{a_{2}}:= w_{(1,0,2)}^{2}w_{(2,2,1)}w_{(2,3,0)}$. We obtain the $I_{+}(\eta)_{4}$-sequence
$$\{w_{(0,0,0)}w_{(1,0,0)}w_{(2,0,5)}w_{(3,5,0)}, w_{(1,0,1)}w_{(1,0,2)}^{2}w_{(3,5,0)},
w_{(1,0,2)}^{2}w_{(2,2,1)}w_{(2,3,0)}\}$$
with  $w_{(0,0,0)}, w_{(3,5,0)}\notin supp(w^{a_2})$.

\vskip 2mm (3) Fix $d=5$ and  consider the non trivial $4$-binomial
$$w_{(0,0,0)}w_{(2,0,5)}w_{(2,2,1)}w_{(2,3,0)}- w_{(1,0,2)}^{3}w_{(3,5,0)}.$$
Since $w_{(0,0,0)}w_{(2,2,1)}- w_{(1,1,0)}w_{(1,1,1)}$ and $w_{(1,0,2)}w_{(3,5,0)} - w_{(2,2,2)}w_{(2,3,0)}$
 are suitable $2$-binomials,
  $\{w_{(0,0,0)}w_{(2,0,5)}w_{(2,2,1)}w_{(2,3,0)}, w_{(1,1,0)}$ $w_{(1,1,1)}w_{(2,0,5)}w_{(2,3,0)}\}$
  and $\{w_{(1,0,2)}^{2}w_{(2,2,2)}w_{(2,3,0)}, w_{(1,0,2)}^{3}w_{(3,5,0)}\}$
  are the $I_{+}(\eta)_{4}$-sequences required in Proposition \ref{Prop:Reductiontype12}.
  Thus is an $I_{+}(\eta)_{4}$-sequence. Furthermore, gluing them  we obtain the $I_{+}(\eta)_{4}$-sequence
$$\{w_{(0,0,0)}w_{(2,0,5)}w_{(2,2,1)}w_{(2,3,0)}, w_{(1,1,0)}w_{(1,1,1)}w_{(2,0,5)}w_{(2,3,0)}, w_{(1,0,2)}^{2}w_{(2,2,2)}w_{(2,3,0)}, w_{(1,0,2)}^{3}w_{(3,5,0)}\}.$$
\end{ex}

We now analyze whether a monomial $m=w_{(r_{1},\gamma_{1},\delta_{1})}w_{(r_{2},\gamma_{2},\delta_{2})}$ with $r_{1},r_{2} \in \{1,2\}$ admits a non-trivial suitable $2$-binomial $m-m'$ with $m' = w_{(r_{3},\gamma_{3},\delta_{3})}w_{(r_{4},\gamma_{4},\delta_{4})}$ and  $r_{3},r_{4} \in \{1,2\}$. This problem can be reformulate as follows. For which $s \geq 0$, setting $\gamma_{3}:= \gamma_{1}\pm s$ and $\gamma_{4}:= \gamma_{2} \mp s$, there exist
$max\{0,(r_{i}-1)d-2\gamma_{i}\} \leq \delta_{i} \leq \lfloor \frac{r_{i}d-3\gamma_{i}}{2} \rfloor$, $i = 3,4$, such that $\delta_{3}+\delta_{4} = \delta_{1}+\delta_{2}$.

\begin{lem}  \label{Lemma:MinMaxBound} With the above notation, there exist such $\delta_{3}$ and $\delta_{4}$ with the following exceptions.
\begin{enumerate}
\item For any  $1\le r_{1},r_{2}\le 2$, if $(r_{1}d_{1}-3\gamma_{1})$ and $(r_{2}d_{2}-3\gamma_{2})$ are even, $s$ is odd, and
$\delta_{1}$ and $\delta_{2}$ are the maximum ones. We call it the {\em maximum bound problem.}
\item Assume $r_{2} = 2$.
\begin{itemize}
\item [(i)] If $r_{1} = 1$, when doing $\gamma_{1}+s$ and $\gamma_{2}-s$
we have $\gamma_{2}-s < k+\varepsilon$ and $\delta_{1}+\delta_{2} < max\{0,d-2\gamma_{2}-2s\}$.
\item [(ii)] If $r_{1} = 2$, when doing $\gamma_{1}+s$ and $\gamma_{2}-s$ we have $\delta_{1} + \delta_{2} < max\{0,d-2\gamma_{1}-2s\} + max\{0,d-2\gamma_{2}+2s\}$ and one of the following cases:
\begin{itemize}
\item [(a)] $\gamma_{1} \geq k+\varepsilon$ and $\gamma_{2} - s < k+\varepsilon$,
\item [(b)] $\gamma_{1}  < k+\varepsilon, \gamma_{1} + s \geq k+\varepsilon, \gamma_{2} \geq k+\varepsilon$ and $\gamma_{1} > \gamma_{2}-s$,
\item [(c)] $\gamma_{1},\gamma_{2} < k+\varepsilon, \gamma_{1}+s > k+\varepsilon$.
\end{itemize}

\end{itemize}
We call it the {\em minimum bound problem.}
\end{enumerate}
\end{lem}

\begin{proof}
We have
$\max\{0,(r_{1}-1)d - 2\gamma_{1}\} + \max\{0,(r_{2}-1)d - 2\gamma_{2}\}
\leq \delta_{1} + \delta_{2} \leq   \lfloor \frac{r_{1}d - 3\gamma_{1}}{2} \rfloor
+ \lfloor \frac{r_{2}d - 3\gamma_{2}}{2} \rfloor$ and
$\max\{0,(r_{1}-1)d - 2(\gamma_{1} + s)\} + \max\{0,(r_{2}-1)d - 2(\gamma_{2} - s)\}
\leq \delta_{3} + \delta_{4} \leq  \lfloor \frac{r_{1}d - 3(\gamma_{1}+s)}{2} \rfloor
+ \lfloor \frac{r_{2}d - 3(\gamma_{2}-s)}{2} \rfloor $.
So the result is clear for those values
$\max\{0,(r_{1}-1)d - 2(\gamma_{1} + s)\} + \max\{0,(r_{2}-1)d - 2(\gamma_{2} - s)\}
\leq \delta_{1} + \delta_{2} \leq  \lfloor \frac{r_{1}d - 3(\gamma_{1}+s)}{2} \rfloor
+ \lfloor \frac{r_{2}d - 3(\gamma_{2}-s)}{2} \rfloor$. Let us study the remainder cases.

(1) From the properties of the floor and ceiling functions we have
$$\lfloor \frac{r_{1}d - 3\gamma_{1}}{2} \rfloor + \lfloor \frac{r_{2}d - 3\gamma_{2}}{2} \rfloor \leq \lfloor \frac{r_{1}d - 3\gamma_{1} + r_{2}d - 3\gamma_{2}}{2} \rfloor = $$
$$\lfloor \frac{r_{1}d - 3(\gamma_{1}+s) + r_{2}d - 3(\gamma_{2}-s)}{2} \rfloor
\leq \lfloor \frac{r_{1}d - 3(\gamma_{1}+s)}{2} \rfloor
+ \lfloor \frac{r_{2}d - 3(\gamma_{2}-s)}{2} \rfloor + 1.$$
Furthermore $\lfloor \frac{r_{1}d - 3(\gamma_{1}+s)}{2} \rfloor
+ \lfloor \frac{r_{2}d - 3(\gamma_{2}-s)}{2} \rfloor < \lfloor \frac{r_{1}d - 3\gamma_{1}}{2} \rfloor + \lfloor \frac{r_{2}d - 3\gamma_{2}}{2} \rfloor$ only when
$(r_{1}d-3\gamma_{1})$ and $(r_{2}d-3\gamma_{2})$ are even and $s$ is odd. From this (1) follows immediately.

(2) is obtained determining which values $\max\{0,(r_{1}-1)d - 2\gamma_{1}\} + \max\{0,(r_{2}-1)d - 2\gamma_{2}\}\leq \delta_{1} + \delta_{2} <
\max\{0,(r_{1}-1)d - 2(\gamma_{1} + s)\} + \max\{0,(r_{2}-1)d - 2(\gamma_{2} - s)\}$.
\end{proof}

\vspace{0.3cm}
Up to here we have proved the following. Suppose given a non trivial suitable $n$-binomial $w^{\alpha} = w^{\alpha_{+}} - w^{\alpha_{-}}$ such that $w^{\alpha_{+}},w^{\alpha_{-}} \notin \mathcal{M}_{3}^{\rho}$. If  $w_{(0,0,0)} \in supp(w^{\alpha})$ or $w_{(3,d,0)} \in supp(w^{\alpha})$, there exit  $I_{+}(\eta)_{n}$-sequences $\{w^{\alpha_{+}}, \hdots, w^{\alpha_{+}'}\}$ and $\{w^{\alpha_{-}'}, \hdots, w^{\alpha_{-}}\}$ such that $w_{(0,0,0)},w_{(3,d,0)} \notin supp(w^{\alpha_{-}'}) \cup supp(w^{\alpha_{-}})$. Clearly $w^{\alpha'}:= w^{\alpha_{+}'} - w^{\alpha_{-}'} \in (I_{+}(\eta))_{n}$. Notice that $w^{\alpha'}$ could be
trivial or even more it could be zero. In the first case $\{w^{\alpha_{+}}, \hdots, w^{\alpha_{+}'}, w^{\alpha_{-}'}, \hdots, w^{\alpha_{-}}\}$ is
an $I_{+}(\eta)_{n}$-sequence. In the other case let $t_{+},t_{-} \geq 0$ be the length of the respective $I_{+}(\eta)_{n}$-sequences. Since $w^{\alpha}$ is non trivial we must have $t_{+} > 0$ or $t_{-} > 0$.
Assume $t_{+} > 0$ (analogously,  for $t_{+} = 0$ and $t_{-} > 0$). Therefore $\{w^{\alpha_{+}}, \hdots, w^{a_{t_{+}-1}}, w^{\alpha_{-}'}, \hdots, w^{\alpha_{-}}\}$ is an $I_{+}(\eta)_{n}$-sequence.
In next Proposition we deal with the case that  $w^{\alpha_{+}'} - w^{\alpha_{-}'}$ is neither trivial nor zero.

\begin{prop} \label{Prop:FinalReduction} Let $w^{\alpha} =
\prod_{i=1}^{t}w_{(1,\gamma_{i},\delta_{i})}\prod_{i=t+1}^{n}w_{(2,\gamma_{i},\delta_{i})}
- \prod_{i=1}^{t}w_{(1,\gamma_{i}',\delta_{i}')}\prod_{i=t+1}^{n}w_{(2,\gamma_{i}',\delta_{i}')}$
be a non trivial suitable $n$-binomial with $n \geq 3$. Therefore, there exist $I_{+}(\eta)_{n}$-sequences $\{w^{\alpha_{+}}, \hdots, w^{\alpha_{+}}_r\}$ and $\{w^{\alpha_{-}}, \hdots, w^{\alpha_{-}}_u\}$ with
$w^{\alpha_{+}}_r = \prod_{i=1}^{t} w_{(1,\gamma_{i}^{1}, \delta_{i}^{1})}\prod_{i=t+1}^{n} w_{(2,\gamma_{i}^{1}, \delta_{i}^{1})}$
and
$w^{\alpha_{-}}_u = \prod_{i=1}^{t} w_{(1,\gamma_{i}^{2}, \delta_{i}^{2})}\prod_{i=t+1}^{n} w_{(2,\gamma_{i}^{2}, \delta_{i}^{2})}$
 satisfying $\gamma_{i}^{1} = \gamma_{i}^{2}$ for all $i = 1,\hdots, n$.
\end{prop}

\begin{proof}
May we assume that $\gamma_{1} \geq \cdots \geq \gamma_{t}, \gamma_{t+1} \geq \cdots \geq \gamma_{n}$ (respectively $\gamma_{i}'$) and let $\gamma_{\ell }$ be first such that $\gamma_{j} \neq \gamma_{j}'$. May we also assume that $\gamma_{\ell } = \gamma_{\ell }'+ s$ with
$s > 0$. Hence $\sum_{j \neq \ell } \gamma_{j} + s = \sum_{j \neq \ell } \gamma_{j}'$. Let $\gamma_{i}$ be the first such that $\gamma_{i} < \gamma_{i}'$ with $i > \ell $ and let $s_{i} > 0$ be such that $\gamma_{i} + s_{i} = \gamma_{i}'$. Now we discuss two cases.

\noindent \underline{Case 1:} $s \leq s_{i}$. According to Lemma \ref{Lemma:MinMaxBound} when doing $\gamma_{\ell } - s$ and $\gamma_{i} + s$ the minimum bound problem (shortly, mbp) does not take place and the maximum bound problem (shortly, MBP) appears when $r_{\ell }d-3\gamma_{\ell }$, $r_{i}d - 3\gamma_{i}$ are even, $s$ is odd, $\delta_{\ell } = \frac{r_{\ell }d - 3\gamma_{\ell }}{2}$ and$\delta_{i} = \frac{r_{i}d-3\gamma_{i}}{2}$. If
MBP does not appear
 we define,
$$w^{a_{2}} = w_{(r_{1}, \gamma_{1},\delta_{1})}w_{(r_{2},\gamma_{2},\delta_{2})}\cdots
w_{(r_{\ell }, \gamma_{\ell } - s, \bar{\delta}_{\ell})}\cdots w_{(r_{i}, \gamma_{i} + s, \bar{\delta}_{i})} \cdots w_{(r_{n},\gamma_{n},\delta_{n})}$$
and then $\{w^{a_{+}}, w^{a_{2}}\}$ is an $I_{+}(\eta)_{n}$-sequence and  $w^{a_{2}}, w^{a_{-}} $ share the same $\gamma $ in position $\ell $.
We assume that the MBP appears and  we divide the discussion in several subcases based on  the parity of $d$.
\begin{itemize}
\item [1.1]  $\varepsilon = 0$,  $\gamma_{\ell }$ and $\gamma_{i}$  even and $s$ odd.
\item [1.1]  $\varepsilon = 1$, $r_{l} = r_{i} = 2$,  $\gamma_{\ell}$ and $\gamma_{i}$  even and $s$ odd.
\item [1.3] $\varepsilon = 1$ , $r_{l} = r_{i} = 1$,  $\gamma_{\ell }$ and $\gamma_{i}$ odd and $s$ odd.
\item [1.4] $\varepsilon = 1$, $r_{l} = 1$, $r_{i} = 2$,  $\gamma_{\ell} $  odd, $\gamma_{i}$  even and $s$ odd.
\end{itemize}

We treat 1.1, the remainder cases are similar and we leave them to the reader. We will modify both  $w^{a_{+}}$ and $w^{a_{-}}$. When doing $\gamma_{\ell } - (s+1)$ and $\gamma_{i} + (s+1)$ the MBP disappears.
Since $\gamma_{\ell}$ and $\gamma_{i}$ are even and $s$ is odd we get that  $\gamma_{\ell}'$ is odd.  If
 $\gamma_{i}' < r_{i}k' + \lfloor \frac{r_{i}\rho}{3} \rfloor$ when doing $\gamma'_{\ell} -1$ and $\gamma '_i+1$
 the mbp does not  appear  and we set
$$w^{a_{2}} = w_{(r_{1},\gamma_{1}, \delta_{1})} \cdots w_{(r_{\ell}, \gamma_{\ell}-(s+1),
\bar{\delta}_{\ell})} \cdots w_{(r_{i}, \gamma_{i} + s+1, \bar{\delta}_{i})} \cdots
w_{(r_{n},\gamma_{n},\delta_{n})}$$
$$w^{a'_{2}} = w_{(r_{1},\gamma_{1}',\delta_{1}')} \cdots w_{(r_{\ell}, \gamma_{\ell}'-1, \bar{\delta}_{\ell}')} \cdots w_{(r_{i}, \gamma_{i}'+1, \bar{\delta}_{i}')} \cdots
w_{(r_{n},\gamma_{n}',\delta_{n}')}.$$
$\{w^{\alpha_{+}},w^{a_{2}}\}$ and $\{w^{a_{2}'},w^{\alpha_{-}'}\}$ are $I_{+}(\eta)_{n}$-sequences and  $w^{a_{2}},w^{a_{2}'}$ share the same $\gamma $ in position $\ell $.

If $\gamma_{i}' = r_{i}k' + \lfloor \frac{r_{i}\rho}{3} \rfloor$, there is no problem when doing $\gamma_{\ell}'+1$, $\gamma_{i}'-1$ and set
$$w^{a_{2}} = w_{(r_{1},\gamma_{1},\delta_{1})} \cdots w_{(r_{\ell},\gamma_{l\ell}(s-1), \bar{\delta}_{\ell})}
 \cdots w_{(r_{i}, \gamma_{i} + (s-1), \bar{\delta}_{i})} \cdots w_{(r_{n}, \gamma_{n}, \delta_{n})}$$
$$w^{a_{2'}} = w_{(r_{1},\gamma_{1},\delta_{1})} \cdots w_{(r_{\ell}, \gamma_{\ell}'+1, \bar{\delta}_{\ell}')} \cdots w_{(r_{i}, \gamma_{i}'-1, \bar{\delta}_{i}')} \cdots w_{(r_{n}, \gamma_{n}', \delta_{n}')}.$$
In any case $w^{a_{2}}$ and $w^{\alpha_{-}}$ (resp. $w^{a_{2}'}$) share the same $\gamma$ in position $\ell$.

\vspace{0.3cm}
\noindent \underline{Case 2:} $s > s_{i}$. Arguing as in Case 1 we distinguish cases 1.1, 1.2, 1.3 and 1.4 and we treat the first one. Assume $\gamma_{\ell}$, $\gamma_{i}$ even, $s$ odd and $\delta_{\ell} = \frac{r_{\ell}d-3\gamma_{\ell}}{2}, \delta_{i} = \frac{r_{i}d-3\gamma_{i}}{2}$. Hence now $\gamma_{i}'$ is odd and we can argue as in Case 1 when doing $\gamma_{\ell}' + 1$ and $\gamma_{i}' - 1$. If $s > s_{i}$, then $w^{a_{2}}$ and $w^{\alpha_{-}}$ (resp. $w^{a_{2}'}$) verifies the same hypothesis that $w^{\alpha_{+}}$ and $w^{\alpha_{-}}$ but now we have $\gamma_{\ell}-s_{i}$ and $\gamma_{\ell}'$ (resp. $\gamma_{\ell}-(s_{i}-1)$ and $\gamma_{\ell}'+1$) in position $\ell$. Next we apply the same strategy to $w^{a_{2}}$ and so on until in step $t > 1$ the resulting monomial $w^{a_{t}}$ verifies Case 1.

The result follows by iterating the above argument.
\end{proof}

\begin{rem} \rm Notice that not necessarily $w^{\alpha_{+}}_{r}-w^{\alpha_{-}}_{u}$
is a non trivial suitable $n$-binomial. In which case we obtain an $I_{n}$-sequence from $w^{\alpha_{+}}$ to $w^{\alpha_{-}}$ arguing as below Proposition \ref{Prop:Reductiontype12}.
\end{rem}

\begin{ex}\rm In Example \ref{Ex:elem4} (2), we had
$w_{(0,0,0)}w_{(1,0,0)}w_{(2,0,5)}w_{(3,5,0)}- w_{(1,1,0)}^{2}w_{(2,1,3)}w_{(2,2,2)}$ and we have build the $I_{+}(\eta)_{4}$-sequence
$$\{w_{(0,0,0)}w_{(1,0,0)}w_{(2,0,5)}w_{(3,5,0)}, w_{(1,0,1)}w_{(1,0,2)}^{2}w_{(3,5,0)},
w_{(1,0,2)}^{2}w_{(2,2,1)}w_{(2,3,0)}\}.$$

Now we apply Proposition \ref{Prop:FinalReduction} to the non trivial $4$-binomial $$w_{(1,1,0)}^{2}w_{(2,1,3)}w_{(2,2,2)}- w_{(1,0,2)}^{2}w_{(2,2,1)}w_{(2,3,0)}.$$

We have $\gamma_{1} = \gamma_{2} = 0, \gamma_{3} = 2, \gamma_{4} = 3$ and
$\gamma_{1}' = \gamma_{2}' = 1, \gamma_{3} = 1, \gamma_{4} = 2$, with $\gamma_{1} = \gamma_{1}'+1$. The first $\gamma_{i} < \gamma_{i}'$ corresponds to $\gamma_{3}$ with $s_{3} = 1$. Then we choose the suitable  $2$-binomial $w_{(1,1,0)}w_{(2,1,3}) - w_{(1,0,1)}w_{(2,2,2)}$ and we define $w^{a_{2}}:= w_{(1,0,1)}w_{(1,1,0)}w_{(2,2,2)}^{2}$. Note the $\gamma$'s involved in
$w^{a_{2}}$ by $\tilde{\gamma}_{i}$, $i = 1,2,3,4$. Now $\tilde{\gamma}_{1} = \gamma_{1}',
\tilde{\gamma}_{2} = 1$, $\tilde{\gamma}_{3} = \tilde{\gamma}_{4} = 2$. The first
$\tilde{\gamma}_{i} > \gamma_{i}'$ is $\gamma_{2} = \gamma_{2}'+1$ and the first $\gamma_{j} < \gamma_{j}'$ with $j \geq 3$ is $\gamma_{4} = 2$ with $s_{4} = 1$. Then we choose the suitable  $2$-binomial $w_{(1,1,0)}w_{(2,2,2)}-w_{(1,0,2)}w_{(2,3,0)}$ and we define
$w^{a_{3}}:= w_{(1,0,1)}w_{(1,0,2)}w_{(2,2,2)}w_{(2,3,0)}$. We have obtained an $I_{+}(\eta)_{4}$-sequence from $w_{(0,0,0)}w_{(1,0,0)}w_{(2,0,5)}w_{(3,5,0)}$ to  $w_{(1,1,0)}^{2}w_{(2,1,3)}w_{(2,2,2)}$. Precisely, $\{w_{(0,0,0)}w_{(1,0,0)}w_{(2,0,5)}w_{(3,5,0)}, w_{(1,0,1)}w_{(1,0,2)}^{2}w_{(3,5,0)},
w_{(1,0,2)}^{2}w_{(2,2,1)}w_{(2,3,0)}, w_{(1,0,1)}w_{(1,0,2)}w_{(2,2,2)}w_{(2,3,0)},$ $ w_{(1,0,1)}w_{(1,1,0)}w_{(2,2,2)}^{2}, w_{(1,1,0)}^{2}w_{(2,1,3)}w_{(2,2,2)}\}.$
\end{ex}

\vspace{0.3cm}
Finally we consider $w^{\alpha} = w^{\alpha_{+}}-w^{\alpha_{-}}$ a non trivial suitable $n$-binomial as in Proposition \ref{Prop:FinalReduction}. Assume that the resulting suitable $n$-binomial $w^{\alpha_{+}}_{r}-w^{\alpha_{-}}_{u} = \prod_{i=1}^{t} w_{(1,\gamma_{i}^{1}, \delta_{i}^{1})}\prod_{i=t+1}^{n} w_{(2,\gamma_{i}^{1}, \delta_{i}^{1})}
- \prod_{i=1}^{t} w_{(1,\gamma_{i}^{2}, \delta_{i}^{2})}\prod_{i=t+1}^{n} w_{(2,\gamma_{i}^{2}, \delta_{i}^{2})}$ is non trivial and non zero. To prove Theorem \ref{Teorema:BigTheorem} it is enough now to show that $w^{\alpha_{+}}_{r}-w^{\alpha_{-}}_{u}$ admits an
$I_{+}(\eta)_{n}$-sequence.
$w^{\alpha_{+}}_{r}-w^{\alpha_{-}}_{u} $ verifies $\gamma_{i}^{1} = \gamma_{i}^{2}$ for all $1 \leq i \leq n$.
For each $\delta_{i}^{1} < \delta_{i}^{2}$, $1 \leq i \leq n$, set $a_{i} = \delta_{i}^{2}- \delta_{i}^{1}$ and $b_{i} = 0$, otherwise set $a_{i} = 0$ and $b_{i} =\delta_{i}^{1} -  \delta_{i}^{2}$. Therefore
$\delta_{1}^{1} + a_{1} - b_{1} + \cdots + \delta_{n}^{1} + a_{n} - b_{n} = \delta_{1}^{2} + \cdots + \delta_{n}^{2}$, which implies that
$a_{1} + \cdots + a_{n} = b_{1} + \cdots + b_{n}$. May we assume that
$a_{1} > 0$. Hence $\delta_{2}^{1} + \cdots + \delta_{n}^{1} > \delta_{2}^{2} + \cdots
+ \delta_{n}^{2}$. Without lost of generality we can assume that $\delta_{i}^{1} > \delta_{i}^{2}$, $i = 2,\hdots, n$. Then $b_{i} > 0$ and $\delta_{i}^{1} + b_{i} = \delta_{i}^{2}$,
$i = 2,\hdots, n$. So $a_{1} \leq b_{2} + \cdots + b_{n}$ and we can consider  $c_{i} \leq b_{i}$ such that $a_{1} = c_{2} + \cdots + c_{n}$. Set
$$w^{a_{2}} = w_{(r_{1},\gamma_{1}^{1}, \delta_{1}^{1} + c_{2})}w_{(r_{2}, \gamma_{2}^{1},
\delta_{2}^{1} - c_{2})} w_{(r_{3}, \gamma_{3}^{1},\delta_{3}^{1})} \cdots
w_{(r_{n}, \gamma_{n}^{1}, \delta_{n}^{1})}.$$
$\{w^{\alpha_{+}}_{r}, w^{a_{2}}\}$ is an $I_{+}(\eta)_{n}$-sequence. If $\delta_{2}^{1}-c_{2} = \delta_{2}^{2}$, then $\{w^{\alpha_{+}}_{r}, w^{a_{2}}, w^{\alpha_{-}}_{u}\}$ is an $I_{+}(\eta)_{n}$-sequence and we finish. Else inductively set
$$2 < i \leq n, \; w^{a_{i}} = w_{(r_{1},\gamma_{1}^{1}, \delta_{1}^{1} + c_{2} + \cdots + c_{i})}w_{(r_{2}, \gamma_{2}^{1}, \delta_{2}^{1} - c_{2})} \cdots w_{(r_{i}, \gamma_{i}^{1},\delta_{i}^{1}-c_{i})} \cdots w_{(r_{i+1}, \gamma_{i+1}^{1}, \delta_{i+1}^{1})} \cdots w_{(r_{n}, \gamma_{n}^{1}, \delta_{n}^{1})}.$$
Since at some point $2 \leq t \leq n$ we achieve $w^{a_{t}}-w^{\alpha_{-}}_{u}$ trivial, we construct an $I_{+}(\eta)_{n}$-sequence $\{w^{\alpha_{+}}_{r}, w^{a_{2}}, \hdots, w^{a_{t}}, w^{\alpha_{-}}_{u}\}$ and the proof of Theorem \ref{Teorema:BigTheorem} is completed.
\hfill $\square$

%%%%%%%%%%%%%%%%%%%%%%%%%%%%%%%%%%%%%%%%%%%%%%
\section{Final remarks and open problems}

In the previous sections we have explicitly described $I(X_d)$. Next goal will be to compute a minimal free resolution of $I(X_d)$ or at least its graded Betti numbers.
Using the program Macaulay2 we have computed a minimal free $R$-resolution of the ideal of the  $GT$-threefold $X_{d}$ for $d = 4,5,6$ and we have got:

\noindent \underline{$d = 4:$}
$$0 \to R(-8)^{9} \to R(-7)^{48} \to R(-6)^{100} \to R(-4)^{6} \oplus
R(-5)^{96} \to $$
$$ R(-3)^{16}\oplus R(-4)^{36} \to R(-2)^{12} \to R \to
R/I_{X_4} \to 0$$

\noindent \underline{$d = 5:$}
$$0 \to R(-10)^{16} \to R(-9)^{120} \to R(-8)^{385} \to R(-7)^{680}
\to R(-6)^{700} \to R(-4)^{15}\oplus R(-5)^{392}  \to $$
$$R(-3)^{48}\oplus R(-4)^{85} \to R(-2)^{20}\oplus R(-3)^{8} \to R \to R/I_{X_5} \to 0$$

\noindent \underline{$d = 6:$}
$$0 \to R(-14)^{4} \oplus R(-15)^{2} \to R(-13)^{108} \oplus R(-14)^{7} \to R(-12)^{803} \to R(-11)^{2850} \to
R(-10)^{6237} \to $$
$$R(-9)^{9064} \to R(-7)^{6} \oplus R(-8)^{8811} \to
R(-6)^{258} \oplus R(-7)^{5352} \to R(-5)^{844} \oplus R(-6)^{1638} \to$$
$$ R(-4)^{796} \oplus R(-5)^{184}\to R(-3)^{322}\oplus R(-4)^{13} \to R(-2)^{57} \to R \to R/I_{X_6} \to 0.$$

\vspace{0.3cm}
It follows from \cite[Proposition 13]{HE} that $X_{d}$ is arithmetically Cohen-Macaulay (see also \cite{BH} and \cite{H}). We would like to address the following problem.

\begin{prob} Find explicitly a minimal free $R$-resolution of $I(X_{d})$ for all $d \geq 4$. 
\end{prob}
%%%%%%%%%%%%%%%%%%%%%%%%%%%%%%%%%%%%%%
%%%%%%%%%%%%%%%%%%%%%%%%%%%%%%%%%%%%%%%%%%%%%%%%%%


\begin{thebibliography}{ll}

\bibitem{AMRV} C. Almeida, Aline V. Andrade and R.M. \ Mir\'o-Roig, {\em Gaps in the number of generators of monomial Togliatti systems}. Journal of Pure and Applied Algebra
{\bf 223} (2019), 1817--1831.

\bibitem{BST} D. Bayer and B. Sturmfels, {\em Cellular resolutions of monomial modules.} J. reine angew. Math. {\bf 502} (1998), 123--140.

\bibitem{BK} H.\ Brenner and A.\ Kaid, {\em Syzygy bundles on $\mathbb P^2$
and the Weak Lefschetz Property},  Illinois J.\ Math.\ {\bf 51}
(2007),  1299--1308.

\bibitem{BH} W. Bruns and J. Herzog, {\em Cohen-Macaulay rings.}
Cambridge University Press, 1993.  

\bibitem{ChaKT} H. Charalambous, A. Katsabekis and A. Thoma, {\em Minimal systems of binomial generators and the indispensable complex of a toric ideal.} Proc. Amer. Math. Soc. {\bf 135}
(2007), 3443--3451.

\bibitem{ChaTV} H. Charalambous, A. Thoma, M. Vladoiu, {\em Binomial fibers and indispensable binomials.} J. Symbolic Computation {\bf 74} (2016), 578--591.

\bibitem{ChaTV1} H. Charalambous, A. Thoma and M. Vladoiu, {\em Minimal generating set of lattice ideals.} M. Collect. Math. {\bf 68} (2017), 377--400.


\bibitem{CMMRS} L. Colarte, E. Mezzetti, R. M. Mir\'{o}-Roig and M. Salat,
\emph{On the coefficients of the permanent and the determinant of a circulant matrix. Applications.} Proc. AMS. {\bf 147} (2019), 547--558.

\bibitem{DST} P. Diaconis and B. Sturmfels, {\em Algebraic Algorithms for sampling from conditional distributions.} The Annals of Statistics {\bf 26} (1998), 363--397.

\bibitem{Grob} W. Gr\"{o}bner, {\em \"{U}ber Veroneseche Variet\"{a}ten und deren Projektionen}, Arch. Math. {\bf 16} (1965), 257 --264.


\bibitem{ES} D. Eisenbud and B. Sturmfels, {\em Binomial ideals}. Duke Math. J. {\bf 84} No. 1 (1996), 1--45.

\bibitem{M} D.R.~Grayson and M.E.~ Stillman,{\em Macaulay2, a software system for research
in algebraic geometry}, {Available at http://www.math.uiuc.edu/Macaulay2/}

\bibitem{HMNW} T.\ Harima, J. Migliore, U. Nagel and J. Watanabe, {\em
The Weak and Strong     Lefschetz Properties for Artinian $K$-Algebras}, J.\
Algebra {\bf 262} (2003), 99-126.

\bibitem{HM} R. Hemmecke and P. Malkin, {\em Computing generating sets of lattice ideals and Markov basis of lattices.} J. Symbolic Computation {\bf 44} (2009), 1463-1476.

\bibitem{H} M. Hochster, {\em Rings of Invariants of Tori, Cohen-Macaulay Rings Generated by Monomials, and Polytopes.} Annals of Mathematics {\bf 96} (1972), 318-337. 

\bibitem{HE} M. Hochster and J. A. Eagon, {\em Cohen-Macaulay Rings, Invariant Theory, and the Generic Perfection of Determinantal Loci}, American Journal of Mathematics {\bf 93} (1971), 1020-1058. 


\bibitem{MeMR} E.\ Mezzetti and R.M. \ Mir\'o-Roig, \emph{The minimal
number of generators of a  Togliatti system}, Annali di Matematica Pura ed Applicata {\bf 195} (2016), 2077-2098. DOI .10.1007/s10231-016-0554-y.

\bibitem{MeMR2} E.\ Mezzetti and R.M. \ Mir\'o-Roig. \emph{Togliatti systems and Galois coverings.} Journal of Algebra {\bf 509} (2018), 263--291.

\bibitem{MMO} E.\ Mezzetti, R.M. \ Mir\'o-Roig and G.\ Ottaviani,
  \emph{Laplace Equations and the Weak Lefschetz Property}, Canad.
J. Math. {\bf 65} (2013), 634--654.


\bibitem{MMR} M.  Micha{\l}ek and R.M.\ Mir\'o-Roig, {\em Smooth
monomial Togliatti systems of cubics}, Journal of
Combinatorial Theory, Ser. A. {\bf 143} (2016), 66--87. http://dx.doi.org
/10.1016/j.jcta.2016.05.004.

\bibitem{MRS}  R.M. \ Mir\'o-Roig and M. Salat, \emph{On the classification of Togliatti systems}, Comm. in Alg. {\bf 46} (2018), 2459--2475.

     \bibitem{T1} E.\ Togliatti, {\em Alcuni esempi di superfici algebriche
     degli iperspazi che rappresentano un'equazione di Laplace},
     Comm. Math. Helvetici {\bf 1} (1929), 255--272.

     \bibitem{T2} E.\ Togliatti, {\em Alcune osservazioni sulle
     superfici razionali che rappresentano equazioni di Laplace},
Ann. Mat. Pura Appl. (4)
      {\bf  25} (1946) 325--339.

\bibitem{ST} B. Sturmfels, {\em Groebner bases of toric varieties.} T\^ohoku Math J.
{\bf 43} (1991), 249-261.


\end{thebibliography}
\end{document}